\documentclass[12pt]{article}
\usepackage[T1]{fontenc}
\usepackage{amssymb}
\usepackage{amsfonts}
\usepackage{amsmath}
\usepackage{mathrsfs}
\usepackage{graphicx}
\usepackage{amsbsy}
\usepackage{theorem}
\usepackage{color}
\usepackage{hyperref}
\usepackage[normalem]{ulem}
\def\Section{\setcounter{equation}{0}\section}
\newtheorem{theorem}{Theorem}[section]
\newtheorem{lemma}[theorem]{Lemma}

\newtheorem{definition}[theorem]{Definition}
\newtheorem{proposition}[theorem]{Proposition}
\newtheorem{remark}[theorem]{Remark} 

\newtheorem{assumption}[theorem]{Assumption}
\def\thetheorem{\thesection.\arabic{theorem}}
\def\thesection{\arabic{section}}

\def\theequation {\thesection.\arabic{equation}}
\def\beq{\begin{equation}\displaystyle}
\def\eeq{\end{equation}}
\def\bel{\begin{equation} \displaystyle \begin{array}{l} }
\def\eel{\end{array} \end{equation} }
\def\bell{\begin{equation} \displaystyle \begin{array}{ll}  }
\def\eell{\end{array} \end{equation} }

\def\bea{\begin{eqnarray}}
\def\eea{\end{eqnarray} }
\def\bean{\begin{eqnarray*}}
\def\eean{\end{eqnarray*} }
\newenvironment{proof}{\noindent{\bf Proof.~}}
{{\mbox{}\hfill {\small \fbox{}}\\}}
\def\qed{\mbox{}\hfill {\small \fbox{}}\\}
\catcode`@=11
\renewcommand\appendix{\bigskip {\noindent \Large \bf Appendix}
  \setcounter{section}{0}%
  \setcounter{subsection}{0}%
\setcounter{equation}{0}%
\setcounter{theorem}{0}%
\def\thetheorem{A.\arabic{theorem}}
\def\theequation {A.\arabic{equation}}}
\catcode`@=12

\def\NN{\mathbb{N}}

\def\RR{\mathbb{R}}

\def\LL{\mathbb{L}}
\def\intdouble {\int\mbox{\hspace{-3mm}}\int}

\def\ds{\displaystyle}

\def\bs{\bigskip}

\def\eps{\varepsilon}

\def\pa{\partial}

\def\calM{{\cal M}}
\def\calW{{\cal W}}
\def\calP{{\cal P}}
\def\calD{{\cal D}}

\def\smes{{\cal S}_{\cal M}}
\def\Dpetit{{\mbox{\tiny $\Delta$}}}
\def\achapo{\widehat{a}}
\def\bchapo{\widehat{b}}

\title{Equivalence between duality and gradient flow solutions 
for one-dimensional aggregation equations}

\date{}

\begin{document}

\maketitle

\begin{center}
F. James$^a$ and N. Vauchelet$^{b}$

\bs

 \medskip

 {\footnotesize $^a$ 
Math\'ematiques -- Analyse, Probabilit\'es, Mod\'elisation -- Orl\'eans (MAPMO), \\
Universit\'e d'Orl\'eans \& CNRS UMR 7349, \\
F\'ed\'eration Denis Poisson, Universit\'e d'Orl\'eans \& CNRS FR 2964, \\
45067 Orl\'eans Cedex 2, France}
 \\

 {\footnotesize $^b$ 
UPMC Univ Paris 06, UMR 7598, Laboratoire Jacques-Louis Lions, \\
CNRS, UMR 7598, Laboratoire Jacques-Louis Lions and \\
INRIA Paris-Rocquencourt, EPI MAMBA \\
F-75005, Paris, France}
 \\

 \medskip
 {\footnotesize {\em E-mail addresses:} 
 {\tt francois.james@univ-orleans.fr, vauchelet@ann.jussieu.fr}
 }

\end{center}

\bs

\begin{abstract}
Existence and uniqueness of global in time measure solution for a
one dimensional nonlinear aggregation equation is considered.
Such a system can be written as a conservation law with a 
velocity field computed through a self-consistent interaction potential.
Blow up of regular solutions is now well established for such 
system. In Carrillo et al. (Duke Math J (2011)) \cite{Carrillo},
a theory of existence and uniqueness based on the geometric
approach of gradient flows on Wasserstein space has been developed.
We propose in this work to establish the link between this approach 
and duality solutions. This latter concept of solutions 
allows in particular to define a flow associated to the velocity field.
Then an existence and uniqueness theory for duality solutions is 
developed in the spirit of James and Vauchelet (NoDEA (2013)) \cite{jamesnv}.
However, since duality solutions are only known in one dimension, we 
restrict our study to the one dimensional case.
\end{abstract}

\bs



\bs

\Section{Introduction}
Aggregation phenomena in a population of particles 
interacting under a continuous interaction potential are modelled by a 
nonlocal nonlinear conservation equation. Letting $\rho$ denote the density of cells,
the so-called aggregation equation in $N$ space dimension writes
	\beq\label{EqInter}
\pa_t\rho + \nabla_x\cdot\big(a(\nabla_xW*\rho) \rho\big) = 0, \qquad t>0,\quad x\in\RR^N,
	\eeq
and is complemented with the inital condition $\rho(0,x)=\rho^{ini}$. Here 
$W\,:\,\RR^N\to\RR$ is the interaction potential, and $a\,:\,\RR^N\to\RR^N$ is a smooth
given function which depends on the actual model under consideration.
In this paper, we only focus on the strongly attractive case
and consider attractive pointy and Lipschitz potentials $W$
(see Definition \ref{pointyPot} below) and nondecreasing smooth function $a$.

This equation is involved in many applications in physics and biology.
In the framework of granular media \cite{benedetto,GME,li}, $a$ is the identity function, and
interaction potentials are in the form $W(x)=-|x|^\alpha$ with $\alpha>1$.
In plasma physics, the context is the high field limit of a kinetic equation describing 
the dynamics of electrically charged Brownian particles interacting
with a thermal bath. This leads to consider potentials in the form $W(x)=-|x|$, and
$a=\mbox{id}$ as well (see e.g. \cite{NPS}).
Also, continuum mathematical models have been widely proposed to model collective
behaviour of individuals. 
Then the potential $W$ is typically the fundamental solution 
of some elliptic equation, and $a$ depends on the microscopic behaviour of the individuals.
In the context of pedestrian motion nonlinear functions $a$ are considered
but with smooth potential $W$
(see \cite{pieton} and later references with generalizations to systems in \cite{pieton2}).
The well-known Patlak-Keller-Segel model describes aggregation of cells 
by a macroscopic non-local interaction equation with linear diffusion \cite{keller,patlack}.
More precisely, the swarming of cells can be described 
by aggregation equations where the typical interaction potential is 
the attractive Morse potential $W(x)=\frac 12 e^{-|x|}$ \cite{burger, morale, okubo}.
Such potentials also appear when considering the hydrodynamic limit of
kinetic model describing chemotaxis of bacteria \cite{dolschmeis,filblaurpert,jamesnv}.

Most of the potentials mentioned above have a singularity at the origin, they fall into the context
of ``pointy potentials'' (see a precise definition below), and
it is well-known that in that case concentration phenomena induce the blow-up in 
finite time of weak $L^p$ solutions \cite{Bertozzi1, Bertozzi2, Bertozzi3}.
Thus the notion of solution breaks down at the blow-up time and weak
measure-valued solutions for the aggregation equation have to be considered
\cite{balague,Bertozzi4, Bertozzi5}.
Carrillo et al. \cite{Carrillo} have studied the multidimensional
aggregation equation when $a=\mbox{id}$ in the framework of gradient flow solutions. 
Namely, equation \eqref{EqInter} is interpreted as a differential equation in time, the
right-hand side being the gradient of some interaction energy defined through the potential $W$.
This idea, known as the Otto calculus (see \cite{Otto,Villani1}), requires the choice of a convenient 
space of probability measures endowed with a Riemannian structure. Then, following
\cite{Ambrosio}, gradient flow solutions are interpreted as curves of maximal slopes in
this space. The authors obtain existence and uniqueness of weak solutions
for \eqref{EqInter} in $\RR^N$, $N\geq 1$ when $a=\mbox{id}$, the main problem being
now to connect these solutions to distributional solutions.

An alternative notion of weak solutions has been obtained by completely different means 
in the framework of positive chemotaxis in \cite{jamesnv}. Here equation \eqref{EqInter}
with $W(x)=\frac 12 e^{-|x|}$ can be obtained thanks to some hydrodynamic 
limit of the kinetic Othmer-Dunbar-Alt system \cite{dolschmeis}.
The key idea is to use the notion of duality solutions, introduced in \cite{bj1} for linear conservation
equations with discontinuous velocities, where measure-valued solutions also arise. In that case, this allows 
to give a convenient meaning to the product of the velocity by the density, so that existence and uniqueness can be proved.
When applying this strategy to the nonlinear case, it turns out that uniqueness is not ensured, unless the nonlinear product
is given a very precise signification, see for instance \cite{BJpg,CRAS}. In the case of chemotaxis, it 
is provided by the limit of the flux in the kinetic model. Once this is done, existence and uniqueness can be obtained.
An important consequence of this approach is that it allows to define a flow associated to this system.
Then the dynamics of the aggregates (i.e. combinations of Dirac masses) can be established, giving rise to
an implementation of a particle method and numerical simulations of the dynamics 
of cells density after blow-up (see also \cite{hyp_proc} for a 
numerical approach using a discretization on a fixed grid).
The principal drawback of this method is its limitation to the one-dimensional case, mainly because duality solutions 
are not properly defined yet in higher space dimension.
Thus the theory developed in \cite{Carrillo}
is, up to our knowledge, the only one allowing to get existence
of global in time weak measure solution for \eqref{EqInter} in 
dimension higher than $2$. Another possibility could be using the notion
of Filippov flow \cite{Filippov}, together with the stability results in  
\cite{Bianchini}, to obtain a convenient notion of solution to \eqref{EqInter},
thus following \cite{Poupaud}.

This work is devoted to the study of the links between these two notions
of weak solutions for equation \eqref{EqInter}, in the one-dimensional setting. 
As we shall see there is no equivalence strictly speaking for a general potential and a nonlinear function $a$.
More precisely we first consider the same situation as in \cite{Carrillo}, that is a pointy potential $W$,
and $a=\mbox{id}$. We adapt the proof of \cite{jamesnv} to define duality solutions in this context, and choose a convenient
space of measures to be compatible with the gradient flows. Then we prove that duality solutions and gradient 
flow solutions are identical (Theorem \ref{th:link} below), thus answering the questions raised by Remark 2.16 of \cite{Carrillo}.

Next, we investigate the nonlinear case, that is $a\ne\mbox{id}$. Notice that additional monotonicity properties
are required to ensure the attractivity of the dynamics.
The results of \cite{Carrillo} cannot be applied as they stand, the key problem is 
to define a new energy for which weak solutions of \eqref{EqInter}
are gradient flows. However, we are able 
to find such an energy only in the particular case $W=- \frac 12 |x|$.
On the contrary existence of duality solutions for \eqref{EqInter} 
with a nonlinear function $a$ can be obtained for more general potentials $W$, even if
we cannot reach the complete generality of \cite{Carrillo}. 
As in the linear case, for this specific choice of potential, we have equivalence
of duality and gradient flows solutions. Moreover, in this case, this solution
can be seen as the derivative of the entropy solution of a scalar conservation law
(Theorem \ref{th:GFanonid} below).

The outline of this paper is as follows.
In the next Section, we introduce notations and recall the main 
results obtained in \cite{Carrillo} in the case $a=\mbox{id}$. A sketch of their proof is proposed.
Section \ref{dualsol} is devoted to the duality solutions, and starts by recalling their original definition and main
properties. Next, we turn to the nonlinear setting, and define precisely the velocities and fluxes
that allow to state the existence and uniqueness results both for $a=\mbox{id}$ and $a\ne\mbox{id}$.
The case $a=\mbox{id}$ is treated in Section \ref{aId}: existence and uniqueness for duality solutions are proved, 
together with equivalence between gradient flows and duality solutions. 
Finally in Section \ref{anonid} we investigate the case $a\neq \mbox{id}$, 
where general equivalence results no longer hold.
For the completeness of the paper, a technical Lemma is given in Appendix.

\Section{Gradient flow solutions}\label{GFsol}
\subsection{Notations and definitions}
Let $C_0(Y,Z)$ be the set of continuous functions from $Y$ to $Z$ 
that vanish at infinity and $C_c(Y,Z)$ the set of continuous 
functions with compact support from $Y$ to $Z$, 
where $Y$ and $Z$ are metric spaces.
All along the paper, we denote $\calM_{loc}(\RR^N)$ the space 
of local Borel measures on $\RR^N$. For $\rho\in {\cal M}_{loc}(\RR^N)$
we denote by $|\rho|(\RR^N)$ its total variation.
We will denote $\calM_b(\RR^N)$ the space of measures in $\calM_{loc}(\RR^N)$
with finite total variation.
From now on, the space of measures $\calM_b(\RR^N)$ is
always endowed with the weak topology $\sigma({\cal M}_b,C_0)$.
We denote $\smes :=C([0,T];{\cal M}_b(\RR^N)-\sigma({\cal M}_b,C_0))$.
We recall that if a sequence of measures $(\mu_n)_{n\in \NN}$ 
in $\calM_b(\RR^N)$ satisfies $\sup_{n\in \NN} |\mu_n|(\RR^N)<+\infty$,
then we can extract a subsequence that converges for the weak topology
$\sigma({\cal M}_b,C_0)$.

Since we focus on scalar conservation laws, we can assume without loss
of generality that the total mass of the system is scaled to $1$ and 
thus we will work in some space of probability measures, namely 
the Wasserstein space of order $q\geq 1$, which is the space
of probability measures with finite order $q$ moment:
	$$
\calP_q(\RR^N) = \left\{\mu \mbox{ nonnegative Borel measure}, \mu(\RR^N)=1, \int |x|^q \mu(dx) <\infty\right\}.
	$$
This space is endowed with the Wasserstein distance defined by (see e.g. \cite{Villani1, Villani2})
	\beq\label{defWp}
d_{Wq}(\mu,\nu)= \inf_{\gamma\in \Gamma(\mu,\nu)} \left\{\int |y-x|^q\,\gamma(dx,dy)\right\}^{1/q}
	\eeq
where $\Gamma(\mu,\nu)$ is the set of measures on $\RR^N\times\RR^N$ with marginals $\mu$ and $\nu$, i.e.
	$$\begin{array}{c}
\Gamma(\mu,\nu) = \displaystyle\left\{\gamma\in \calP_q(\RR^N\times\RR^N); \ \forall\, \xi\in C_0(\RR^N),
\int \xi(y_0)\gamma(dy_0,dy_1) = \int \xi(y_0) \mu(dy_0), \right.\\[2mm]
\displaystyle\left.\int \xi(y_1)\gamma(dy_0,dy_1) = \int \xi(y_1) \nu(dy_1) \right\}.
	\end{array}$$
From a simple minimization argument, we know that in the definition of $d_{Wq}$ 
the infimum is actually a minimum.
A map that realizes the minimum in the definition \eqref{defWp}
of $d_{Wq}$ is called an optimal map, the set of which is denoted by $\Gamma_0(\mu,\nu)$.

A fundamental breakthrough in the use of the geometric approach to 
solve PDE is the work of F. Otto \cite{Otto}, which is the basis
of the so-called {\it Otto Calculus} \cite{Villani1}.
Let $X$ be a Riemannian manifold endowed with the Riemannian metric 
$g_x(\cdot,\cdot)$ (a positive quadratic form on the tangent space at $X$ in $x$ denoted $T_xX$).
Let $\calW : X\to \RR$ be differentiable. The gradient of $\calW$ at $x\in X$ is 
defined as follows~: for all $v\in T_xX$, let $\gamma(t)$ be a regular curve on $X$ 
such that $\gamma(0)=x$ and $\gamma'(0)=v$, then
	$$
\frac{d}{dt}|_{t=0} \calW(\gamma)(t) = g_x(\nabla_x \calW,v), \qquad 
\nabla_x\calW \in T_xX.
	$$
The gradient flow associated to $\calW$ is the solution
$\rho:[0,+\infty)\to X$ of the differential equation~:
	$$
\frac{d\rho}{dt} = -\nabla_\rho \calW.
	$$
A fundamental result due to Ambrosio et al. \cite{Ambrosio} states
that gradient flows are equivalent to curves of maximal slope. 
Therefore, solving a PDE model of gradient type boils down to 
prove the existence of a curve of maximal slope.

Let us be more precise.
In the following, we will mainly focus on the case $q=2$ and we will
shortly denote $d_W$ instead of $d_{W2}$.
In the formalism of \cite{Ambrosio},
we say that a curve $\mu$ is absolutely continuous, and we denote
$\mu \in AC^2((0,+\infty),\calP_2(\RR^N))$, if there exists $m \in L^2(0,+\infty)$, such that 
$d_W(\mu(s),\mu(t))\leq \int_s^t m(r)dr$, for $0<s\leq t<+\infty$.
Then we can define the metric derivative
	$$
|\mu'|(t) := \limsup_{s\to t} \frac{d_W(\mu(s),\mu(t))}{|s-t|}.
	$$
The tangent space to a measure $\mu\in \calP_2(\RR^N)$ is defined by
the closed vector subspace of $L^2(\mu)$ 
	$$
Tan_\mu \calP_2(\RR^N) := \overline{\{\nabla \phi : \phi \in C_c^\infty(\RR^N)\}}^{L^2(\mu)}.
	$$

We recall the result of Theorem 8.3.1 of \cite{Ambrosio}:
if $\mu \in AC^2((0,+\infty),\calP_2(\RR^N))$, then there exists a Borel
vector field $v(t)\in L^2(\mu(t))$ such that
	\beq\label{EqCont}
\pa_t\mu + \nabla\cdot (v \mu) = 0, \mbox{ in the distributional sense on } (0,+\infty)\times \RR^N.
	\eeq
Conversely, if $\mu$ solves a continuity equation for some Borel velocity
$v\in L^1((0,+\infty);L^2(\mu))$ then $\mu$ is an absolutely continuous
curve and $|\mu'|(t)\leq \|v(t)\|_{L^2(\mu)}$.
As a consequence, we have
	$$
|\mu'|(t) = \min\left\{\|v\|_{L^2(\mu(t))}: v(t)\mbox{ satisfies }\eqref{EqCont}\right\}.
	$$

Let $\calW$ be a functional on $\calP_2(\RR^N)$.
We denote by $\partial\calW$ its subdifferential. The general definition of subdifferential 
on $\calP_2(\RR^d)$ being pretty involved, we refer the interested reader 
to \cite[Definition 10.3.1]{Ambrosio}.
In principle, the element of $\partial\calW$ are plans. In the case at hand, 
such plans are concentrated on the graph of a vector field, which allows to 
reduce the general definition of subdifferential to the following one:
a vector field $w\in L^2(\mu)$ is said to be an element of the subdifferential of $\calW$
at $\mu$ if
$$
\calW(\mu)-\calW(\nu) \geq \inf_{\gamma\in \Gamma_0(\mu,\nu)} \int_{\RR^N\times \RR^N} 
w(x)\cdot (y-x) d\gamma(x,y) + o(d_{W2}(\mu,\nu)).
$$
Next, the slope $|\pa \calW|$ is defined by
	\beq\label{slope}
|\pa \calW|(\mu) = \limsup_{\nu\to \mu} \frac{(\calW(\mu)-\calW(\nu))_+}{d_{W}(\mu,\nu)},
	\eeq
where $u_+=\max\{u,0\}$. We have the property (\cite{Ambrosio}, Lemma 10.1.5)
	\beq\label{slopmin}
|\pa\calW|(\mu) = \min \{\|w\|_{L^2(\mu)}\,:~ w\in \pa\calW(\mu)\}.
	\eeq
Moreover, there exists a unique $w\in\calW(\mu)$ which attains the minimum in \eqref{slopmin}.
It is denoted by $\partial^0\calW(\mu)$.

\begin{definition}[Gradient flows]\label{defgradflow}
We say that a map $\mu\in AC^2_{loc}((0,+\infty);\calP_2(\RR^N))$
is a solution of a gradient flow equation associated to the functional
$\calW$ if there exists a Borel vector field $v$ such that $v(t)\in Tan_{\mu(t)}\calP_2(\RR^N)$ for a.e. $t>0$,
$\|v(t)\|_{L^2(\mu)} \in L^2_{loc}(0,+\infty)$, the continuity equation 
	$$
\pa_t \mu + \nabla\cdot \big(v \mu \big)=0,
	$$
holds in the sense of distributions and $v(t)\in -\pa \calW(\rho(t))$ for a.e. 
$t>0$, where $\pa\calW(\rho)$ is the subdifferential of $\calW$ at the point 
$\rho$.
\end{definition}

	\begin{definition}[Curve of maximal slope]\label{defmaxslop}
A curve $\mu\in AC^2_{loc}((0,+\infty);\calP_2(\RR))$ is a curve of 
maximal slope for the functional $\calW$ if $t\mapsto \calW(\mu(t))$ is 
an absolutely continuous function and if for every $0\leq s\leq t\leq T$,
	$$
\frac 12 \int_s^t |\mu'|^2(\tau)\,d\tau + \frac 12 \int_s^t |\pa\calW|^2(\mu(\tau)) d\tau
\leq \calW(\mu(s))-\calW(\mu(t)).
	$$
	\end{definition} 

Finally Theorem 11.1.3 of \cite{Ambrosio} shows that curves of maximal slope with respect
to $|\partial\calW|$ are equivalent to gradient flow solutions. 
Moreover, the tangent vector field $v(t)$ is the unique element of
minimal norm in the subdifferential of $\calW$ (see \eqref{slopmin}):
	$$
v(t)=-\pa^0 \calW(\mu(t)) \mbox{ for a.e. } t>0.
	$$

\subsection{Strategy of the proof in \cite{Carrillo}}\label{sec:Carrillo}
The idea of the work by Carrillo et al. \cite{Carrillo} is to extend the
work of \cite{Ambrosio} to an interaction energy $\calW$ defined 
through the interaction potential $W$ in \eqref{EqInter}, whose derivative has a 
singularity in $0$. More precisely, attractive ``pointy potentials'' are considered, which we define now.
	\begin{definition}[pointy potential]\label{pointyPot}
The interaction potential $W$ is said to be an attractive 
{\it pointy potential} if it satisfies the following assumptions.
\begin{itemize}
\item[{\bf (A0)}] $W$ is continuous, $W(x)=W(-x)$ and $W(0)=0$.
\item[{\bf (A1)}] $W$ is $\lambda$-concave for some $\lambda\geq 0$, i.e.
  $W(x)-\frac{\lambda}{2}|x|^2$ is concave.
\item[{\bf (A2)}] There exists a constant $C>0$ such that $W(x)\geq -C(1+|x|^2)$
for all $x\in\RR^N$.
\item[{\bf (A3)}] $W\in C^1(\RR^N\setminus\{0\})$.  
\end{itemize}
	\end{definition}
Given a continuous potential $W:\RR\to \RR$, we define the interaction energy in one dimension by
\beq\label{nrjinter}
\calW(\rho)= -\frac 12\int_{\RR\times\RR} W(x-y)\,\rho(dx)\rho(dy).
\eeq
The existence and uniqueness result of \cite{Carrillo} can now be synthetized as follows.
	\begin{theorem}\label{GradFlow}(\cite[Theorems 2.12 and 2.13]{Carrillo})
Let $W$ satisfies assumptions {\bf (A0)--(A3)} and let $a=\mbox{id}$.
Given $\rho^{ini}\in \calP_2(\RR^N)$,
there exists a gradient flow solution of \eqref{EqInter}, 
i.e. a curve $\rho\in AC_{loc}^2([0,\infty);\calP_2(\RR^N))$ satisfying
	$$
\begin{array}{l}
\ds \frac{\pa \rho(t)}{\pa t} + \pa_x (v(t)\rho(t))=0, \qquad \mbox{in }\calD'([0,\infty)\times \RR^N), \\
\ds v(t,x)=-\pa^0 \calW(\rho)(t,x) = \int_{y\neq x} \nabla W(x-y)\,\rho(t,dy),
\end{array}
	$$
with $\rho(0)=\rho^{ini}$. Moreover, if $\rho_1$ and $\rho_2$
are such gradient flow solutions, then there exists a constant $\lambda$
such that, for all $t\geq 0$
	$$
d_W(\rho_1(t),\rho_2(t)) \leq e^{\lambda t} d_W(\rho_1(0),\rho_2(0)).
	$$
Thus the gradient flow solution of \eqref{EqInter} with initial data $\rho^{ini}\in \calP_2(\RR^N)$ is unique.
Moreover, the following energy identity holds for all $0\leq t_0\leq t_1 < \infty$:
	\begin{equation}\label{EstimEnerg}
\int_{t_0}^{t_1} \int_\RR |\partial^0W*\rho|^2 \rho(t,dx)dt + \calW(\rho(t_1)) = \calW(\rho(t_0)).
	\end{equation}
\end{theorem}

\begin{proof} We summarize here the main steps of the proof and 
refer the reader to \cite{Carrillo} for more details.
The first step is to compute the element of minimal norm in the 
subdifferential of $\calW$. By extending Theorem 10.4.11 of
\cite{Ambrosio}, the authors prove \cite[Proposition 2.6]{Carrillo} that
	\beq\label{sousdifW}
-\pa^0\calW(\rho) = \pa^0 W * \rho, 
	\eeq
where 
	$$
\pa^0 W * \rho(x)= \int_{y\neq x} W'(x-y)\,\rho(dy).
	$$

The second step is based on the so-called {\it JKO scheme} introduced in \cite{JKO}
(see also \cite{Ambrosio}). It consists in the following recursive construction for curves of maximal slope.
Let $\tau>0$ be a small time step, we set $\rho_0^\tau = \rho^{ini}$ the initial data for \eqref{EqInter}.
Next, knowing $\rho_k^\tau$, one proves \cite[Proposition 2.5]{Carrillo} that there exists $\rho_{k+1}^\tau$ such that
	\beq\label{JKO}
\rho_{k+1}^{\tau} \in \underset{\rho\in \calP_2(\RR^N)}{\arg\min} \left\{
\calW(\rho) + \frac{1}{2\tau} d_W^2(\rho_k^\tau,\rho)\right\}.
	\eeq
Next, a piecewise constant interpolation $\rho^\tau$ is defined by 
	$$
\rho^\tau(0)=\rho^{ini}\ ;\qquad \rho^\tau(t)=\rho_k^\tau 
\quad \mbox{ if } \quad t\in (k \tau,(k+1)\tau],
	$$
and Proposition 2.6 states the weak compactness (in the narrow topology) of the sequence $\rho^\tau$ as 
$\tau\to 0$.
Finally Theorem 2.8 ensures that the weak narrow limit $\rho$ is a curve of maximal slope.

The conclusion follows by applying Theorem 11.1.3 of \cite{Ambrosio}, which allows therefore to get the existence 
of a gradient flow for the functional $\calW$. By definition, the gradient flow is a 
solution of a continuity equation whose velocity field is the 
element with minimal norm of the subdifferential of $\calW$. In the first
step of the proof this element has been identified to be $\pa^0W*\rho$.
Thus, it is a weak solution of the problem and moreover we have the 
energy estimate \eqref{EstimEnerg}.
\end{proof}

\subsection{The one-dimensional case}\label{sec:1D}
We gather here several remarks specific to the one-dimensional framework. First we notice that 
assumptions {\bf (A1)} and {\bf (A3)} imply that
$x\mapsto W'(x)-\lambda x$ is a nonincreasing function on $\RR\setminus\{0\}$.
Therefore $\lim_{x\to 0^{\pm}} W'(x) = W'(0^{\pm})$ exists and from 
{\bf (A0)}, we deduce that $W'(0^-)=-W'(0^+)$.
Moreover, for all $x>y$ in $\RR\setminus\{0\}$ we have 
$W'(x)-\lambda x\leq W'(y)-\lambda y$. 
Thus we have the one-sided Lipschitz estimate (OSL) for $W'$
	\begin{equation}\label{WOSL}
\forall\, x>y\in \RR\setminus\{0\}, \qquad  W'(x)-W'(y)\leq \lambda (x-y).
	\end{equation}
Letting $y\to 0^{\pm}$ we deduce that for all $x>0$, 
$W'(x)-\lambda x \leq W'(0^+)$ and $W'(x)-\lambda x \leq W'(0^-)$. 
Thus we also have the one-sided estimate
	\begin{equation}  \label{WOSL1}
W'(x) \leq \lambda x, \quad \mbox{ for all } x>0.  
	\end{equation}

In the following, we will assume that the potential $W$ is Lipschitz
in order to avoid possible linear growth at infinity of the velocity field~:
\begin{itemize}
\item[{\bf (A4)}] There exists a nonnegative constant $C_0$ such that $|W'(x)|\leq C_0$ for all $x\in \RR^*$.
\end{itemize}


The one-dimensional framework also allows to simplify several proofs in Theorem \ref{GradFlow}.
Indeed any probability measure $\mu$ on the real line $\RR$
can be described in term of its cumulative distribution function
$M(x)=\mu((-\infty,x))$ which is a right-continuous
and nondecreasing function with $M(-\infty)=0$ and $M(+\infty)=1$.
Then we can define the generalized inverse $F_\mu$ of $M$ (or monotone rearrangement of $\mu$)
by $F_\mu(z)=M^{-1}(z):=\inf\{x\in \RR / M(x)>z\}$, it is a right-continuous
and nondecreasing function as well, defined on $[0,1]$.
We have for every nonnegative Borel map $\xi$,
$$
\int_\RR \xi(x) \rho(dx) = \int_0^1 \xi(F_\mu(z))\,dz.
$$
In particular, $\mu\in \calP_2(\RR)$ if and only if $F_\mu\in L^2(0,1)$.
Moreover, in the one-dimensional setting, there exists a unique optimal map
realizing the minimum in \eqref{defWp}. 
More precisely, if $\mu$ and $\nu$ belongs to $\calP_2(\RR)$, with monotone rearrangement
$F_\mu$ and $F_\nu$, then $\Gamma_0(\mu,\nu)=\{(F_\mu,F_\nu)_\# {\LL}_{(0,1)}\}$ where 
${\LL}_{(0,1)}$ is the restriction of the Lebesgue measure on $(0,1)$.
Then we have the explicit expression of the Wasserstein distance (see \cite{rachev,Villani2}) 
	\begin{equation}\label{dWF-1}
d_W(\mu,\nu)^2 = \int_0^1 |F_\mu(z)-F_\nu(z)|^2\,dz,
	\end{equation}
and the map $\mu \mapsto F_\mu$ is an isometry between $\calP_2(\RR)$ and the convex subset
of (essentially) nondecreasing functions of $L^2(0,1)$.

In this framework, we can then rewrite the JKO scheme \eqref{JKO} 
in the proof above. Let us denote by 
$M_k^{\tau}$ the cumulative distribution of the measure $\rho_k^{\tau}$ and
by $V_k^{\tau}:=F_k^{\tau}$ its generalized inverse.
Then, in term of generalized inverses, \eqref{JKO} rewrites
$$
V_{k+1}^\tau \in \underset{\{V\in L^2(0,1), \pa_xV \geq 0\}}{\arg\min} \left\{ 
\widetilde{\calW}(V) + \frac{1}{2\tau} \|V-V_k^{\tau}\|_{L^2(0,1)}^2 \right\},
$$
where
$$
\widetilde{\calW}(V) = \frac 12 \int_0^1\int_0^1 W(V(y)-V(z))\,dydz.
$$
Such an approach using the generalized inverse has been used in \cite{BCC}
for the one dimensional Patlak-Keller-Segel equation.

\Section{Duality solutions}\label{dualsol}
We turn now to the alternative notion of weak solution we wish to investigate. 
It is based on the so-called duality solutions which were introduced for 
linear advection equations with discontinuous coefficients in \cite{bj1}. 
Compared with the gradient flow approach, this strategy allows
a more straightforward PDE formulation.
In particular from the numerical viewpoint, classical finite volume approach strongly relying on this formulation is proposed in \cite{numerix}.
The main drawback is that presently duality solutions in any space dimension 
are only available for pure transport equations
(see \cite{bujama}). Since we have to deal here with conservative balance laws,
we have to restrict ourselves to one space dimension. 
First we give a brief account of the theory developed in \cite{bj1}, 
summarizing the main theorems we shall use, 
next we define duality solutions for \eqref{EqInter}.

\subsection{Linear conservation equations}\label{duality}
We consider here conservation equations in the form
	\beq\label{eq.conserve}
\partial_t\rho + \pa_x\big(b(t,x)\rho\big) = 0, \qquad (t,x)\in(0,T)\times\RR,
	\eeq
where $b$ is a given bounded Borel function. Since no regularity is assumed for $b$, 
solutions to \eqref{eq.conserve} eventually are measures in space. A convenient tool to handle this
is the notion of duality solutions, which are defined as weak solutions, the test functions 
being Lipschitz solutions to the backward linear transport equation
	\begin{eqnarray}
& & \ds \partial_t p + b(t,x) \partial_x p = 0, \label{(3)}\\
& & \ds p(T,.) = p^T \in {\rm Lip}(\RR).
	\end{eqnarray}
In fact, a formal computation shows that 
$\frac{d}{dt}\left(\int_{\RR}p(t,x)\rho(t,dx)\right) = 0$,
which defines the duality solutions for suitable $p$'s. 

It is quite classical that a sufficient condition to ensure existence for \eqref{(3)} is that 
the velocity field to be compressive, in the following sense:
	\begin{definition}
We say that the function $b$ satisfies the one-sided Lipschitz (OSL) condition if
	\begin{equation}\label{OSLC}
\partial_x b(t,.)\leq \beta(t)\qquad\mbox{for $\beta\in L^1(0,T)$ in the distributional sense}.
	\eeq
	\end{definition}
However, to have uniqueness, we need to restrict ourselves to {\it reversible} solutions of (\ref{(3)}): let $\mathcal L$ 
denote the set of Lipschitz continuous solutions to (\ref{(3)}), and define
the set ${\mathcal E}$ of {\it exceptional solutions} by
	$$
{\mathcal E} = \Big\{p\in{\mathcal L} \mbox{ such that } p^T \equiv 0 \Big\}.
	$$ 
The possible loss of uniqueness corresponds to the case where ${\mathcal E}$ is not reduced to zero.
	\begin{definition}
We say that $p\in{\mathcal L}$ is a {\bf reversible solution} to (\ref{(3)}) 
if $p$ is locally constant on the set 
	$$
{\mathcal V}_e=\Big\{(t,x)\in [0,T] \times \RR;\ \exists\ p_e\in{\mathcal E},\ p_e(t,x)\not=0\Big\}.
	$$
	\end{definition}
We refer to \cite{bj1} for complete statements of the characterization and properties of reversible solutions.
Then, we can state the definition of duality solutions.
	\begin{definition}
We say that 
$\rho\in \smes := C([0,T];{\cal M}_b(\RR)-\sigma({\cal M}_b,C_0))$
is a {\bf duality solution} to (\ref{eq.conserve}) 
if for any $0<\tau\le T$, and any {\bf reversible} solution $p$ to (\ref{(3)})
with compact support in $x$,
the function $\displaystyle t\mapsto\int_{\RR}p(t,x)\rho(t,dx)$ is constant on
$[0,\tau]$.
	\label{defdual}\end{definition}

We summarize now some properties of duality solutions
that we shall need in the following.

	\begin{theorem}(Bouchut, James \cite{bj1})\label{ExistDuality}
\begin{enumerate}
\item Given $\rho^\circ \in {\cal M}_b(\RR)$, under the assumptions
(\ref{OSLC}), there exists a unique $\rho \in \smes$,
duality solution to (\ref{eq.conserve}), such that $\rho(0,.)=\rho^\circ$. \\
Moreover, if $\rho^\circ$ is nonnegative, then $\rho(t,\cdot)$ is nonnegative
for a.e. $t\geq 0$. And we have the mass conservation 
	$$
|\rho(t,\cdot)|(\RR) = |\rho^\circ|(\RR), \quad \mbox{ for a.e. } t\in ]0,T[.
	$$
\item Backward flow and push-forward: the duality solution satisfies
	\beq\label{flow}
\forall\, t\in [0,T], \forall\, \phi\in C_0(\RR),\quad
\int_\RR \phi(x)\rho(t,dx) = \int_\RR \phi(X(t,0,x)) \rho^0(dx),
	\eeq
where the {\bf backward flow} $X$ is defined as the unique reversible
solution to
	$$
\pa_tX + b(t,x) \pa_xX = 0 \quad \mbox{ in } ]0,s[\times\RR, \qquad
X(s,s,x)=x.
	$$
\item For any duality solution $\rho$, we define the {\bf generalized flux}
  corresponding to $\rho$ by 
$b\Dpetit \rho = -\pa_t u$, where $u=\int^x \rho\,dx$.

There exists a bounded Borel function $\widehat b$, called {\bf
universal representative} of $b$, such that $\widehat b = b$
almost everywhere, $b\Dpetit \rho =\widehat b\rho$ and for any duality solution $\rho$,
	$$
\partial_t \rho + \partial_x(\widehat b\rho) = 0 \qquad \hbox{in the distributional sense.}
	$$
\item Let $(b_n)$ be a bounded sequence in
$L^\infty(]0,T[\times\RR)$, such that
$b_n\rightharpoonup b$ in $L^\infty(]0,T[\times\RR)-w\star$. Assume
$\partial_x b_n\le \beta_n(t)$, where $(\beta_n)$ is bounded in $L^1(]0,T[)$,
$\partial_x b\le\beta\in L^1(]0,T[)$.
Consider a sequence $(\rho_n)\in\smes$ of duality solutions to
	$$
\partial_t\rho_n+\partial_x(b_n\rho_n)=0\quad\hbox{in}\quad]0,T[\times\RR,
	$$
such that $\rho_n(0,.)$ is bounded in ${\cal M}_{b}(\RR)$, and
$\rho_n(0,.)\rightharpoonup\rho^\circ\in{\cal M}_{b}(\RR)$.

\noindent Then $\rho_n\rightharpoonup \rho$ in $\smes$, where $\rho\in\smes$ is the
duality solution to
	$$
\partial_t\rho+\partial_x(b\rho)=0\quad\hbox{in}\quad]0,T[\times\RR,\qquad
\rho(0,.)=\rho^\circ.
	$$
Moreover, $\widehat b_n\rho_n\rightharpoonup \widehat b\rho$ weakly in ${\cal M}_{b}(]0,T[\times\RR)$.
\end{enumerate}
	\end{theorem}

The set of duality solutions is clearly a vector space, but it has
to be noted that a duality solution is not {\it a priori} defined as 
a solution in the sense of distributions. 
However, assuming that the coefficient $b$ is piecewise continuous,
we have the following equivalence result:
	\begin{theorem}\label{dual2distrib}
Let us assume that in addition to the OSL condition \eqref{OSLC}, 
$b$ is piecewise continuous on $]0,T[\times\RR$ where the 
set of discontinuity is locally finite.
Then there exists a function $\bchapo$ which coincides with $b$
on the set of continuity of $b$.

With this $\bchapo$, $\rho\in \smes$ is a duality solution to 
\eqref{eq.conserve} if and only if $\pa_t\rho+\pa_x(\bchapo\rho)=0$
in ${\mathcal D}'(\RR)$. Then the generalized flux 
$b\Dpetit \rho = \bchapo \rho$. 
In particular, $\bchapo$ is a universal representative of $b$.
	\end{theorem}

This result comes from the uniqueness of solutions to the Cauchy 
problem for both kinds of solutions, see \cite[Theorem 4.3.7]{bj1}.

\subsection{Duality solutions for aggregation}\label{dual:agreg}
Equipped with this notion of solutions, we can now define duality solutions for the aggregation equation.
The idea was introduced in \cite{BJpg} in the context of pressureless gases. It was next applied to chemotaxis in \cite{jamesnv},
and we shall actually follow these steps.
	\begin{definition}\label{defexist}
We say that $\rho\in \smes$
is a duality solution to \eqref{EqInter} if there exists 
$\widehat{a}_\rho\in L^\infty((0,T)\times\RR)$ and $\beta\in L^1_{loc}(0,T)$ 
satisfying $\pa_x\widehat{a}_\rho\le \beta$ in $\calD'((0,T)\times\RR)$, 
such that for all $0<t_1<t_2<T$,
	\begin{equation}\label{rhodis}
\pa_t\rho + \pa_x(\widehat{a}_\rho\rho) = 0
	\end{equation}
in the sense of duality on $(t_1,t_2)$,
and $\widehat{a}_\rho=a(W'*\rho)$ a.e.
We emphasize that it means that the final datum for \eqref{(3)} should be at $t_2$ instead of $T$.
	\end{definition}
This allows at first to give a meaning to the notion of distributional solution, but it turns out that uniqueness is 
a crucial issue. For that, a key point is a precise definition of the product $\widehat{a}_\rho\rho$, as we shall see in
more details in Section \ref{vel:flux} below.

We now state the main theorems about duality solutions for the aggregation equation \eqref{EqInter}. 
Existence of such solutions in a measure space has been 
obtained in \cite{jamesnv} in the particular case $W(x)=\frac 12 e^{-|x|}$ and
a similar result is presented in \cite{CRAS} when $W(x)=-|x|/2$ which 
appears in many applications in physics or biology.
We extend here these results for a general potential satisfying
assumptions {\bf (A0)--(A4)}. However to do so, we have, as in 
\cite{Carrillo}, to restrict ourselves to the linear case,
that is $a=\mbox{id}$.
	\begin{theorem}[Duality solutions, linear case]\label{th:duality}
Let $W$ satisfy assumptions {\bf (A0)--(A4)} and $a=\mbox{id}$.
Assume that $\rho^{ini}\in \calP_1(\RR)$.
Then for any $T>0$, there exists a unique $\rho\in\smes$ such that $\rho(0)=\rho^{ini}$, 
$\rho(t)\in \calP_1(\RR)$ for any $t\in (0,T)$, and $\rho$ is a duality solution
to equation \eqref{EqInter} with universal representative $\widehat{a}_\rho$ in \eqref{rhodis} defined by
	\begin{equation}\label{achapo}
\widehat{a}_\rho(t,x):= \pa^0W*\rho(t,x) = \int_{x\neq y} W'(x-y)\rho(t,dy).
	\end{equation}
Moreover we have $\rho = X_\# \rho^{ini}$ where $X$ is the 
backward flow corresponding to $\widehat{a}_\rho$.
	\end{theorem}

We turn now to the case $a\neq \mbox{id}$. In order to be in the attractive case, 
we assume that the following
	\begin{assumption}\label{assump1}
The potential $W$ is Lipschitz and pointy, i.e. satisfies {\bf (A0)--(A4)}.
The function $a$ is non-decreasing, with
	\beq\label{hyp_a}
a\in C^1(\RR),\quad 0\leq a' \leq \alpha, \quad \alpha>0.
	\eeq
	\end{assumption}
In this context, existence and uniqueness of duality solutions
have been proved for the case of an interaction potential 
$W=\frac 12 e^{-|x|}$ in \cite{jamesnv}.
We extend here the techniques developed in this latter work to 
more general potentials $W$. However we are not able to prove such results in the whole generality of 
assumptions {\bf (A0)--(A4)} and need more regularity on the 
interaction potential, as follows
	\begin{assumption}\label{assump}
We assume that $W\in C^1(\RR\setminus\{0\})$ and that in the distributional sense
	\beq\label{hyp}
W''= -\delta_0 + w, \quad w\in  Lip\cap L^\infty(\RR),
	\eeq
where $\delta_0$ is the Dirac measure in $0$.
	\end{assumption}

This allows a definition of the flux in \eqref{EqInter} which generalizes the one in \cite{jamesnv}. Indeed we can formally
take the convolution of \eqref{hyp} by $\rho$, then multiply by $a(W'*\rho)$. Denoting by 
 $A$ the antiderivative of $a$ such that $A(0)=0$  and using the chain rule we obtain formally
	\begin{equation}\label{ChainRule}
-\pa_x(A(W'*\rho))=-a(W'*\rho)W''*\rho = a(W'*\rho)(\rho-w*\rho).
	\end{equation}
Thus a natural formulation for the flux $J$ is given by
	\beq\label{DefFluxJ}
J:= -\pa_x\big(A(W'*\rho)\big)+a(W'*\rho)w*\rho.
	\eeq
The product $a(W'*\rho)w*\rho$ is well defined since $w*\rho$ is Lipschitz.
Then $J$ is defined in the sense of measures.

	\begin{theorem}[duality solutions, nonlinear case]\label{dual_anonid}
Let be given $\rho^{ini}\in\calP_1(\RR)$.
Under Assumptions \ref{assump1} and \ref{assump} on the potential $W$ and
the nonlinear function $a$, for all $T> 0$ there exists
a unique duality solution $\rho$ of \eqref{EqInter} in the sense 
of Definition \ref{defexist} $\rho(t)\in \calP_1(\RR)$
for $t\in (0,T)$ and which satisfies in the distributional sense
	\beq\label{eqrhodis}
\pa_t \rho + \pa_x J = 0,
	\eeq
where $J$ is defined by \eqref{DefFluxJ}. 
	\end{theorem}
Theorems \ref{th:duality} and \ref{dual_anonid} are proved respectively in Sections \ref{aId} and \ref{anonid}, but
before diving into the detailed proofs, we comment the main steps, which are common to both cases.
\begin{itemize}
\item {\bf existence} of duality solutions is obtained by approximation. 
First we obtain the dynamics of aggregates (that is combinations of Dirac masses), 
then we proceed using the stability of duality solutions
\item {\bf uniqueness} is obtained by a contraction argument in $\calP_1(\RR)$.
No uniqueness is expected in a general space of measures. The argument is repeated in $\calP_2(\RR)$ so that gradient flow and duality solutions can be compared.
In the nonlinear case, the contraction argument 
relies on an entropy inequality.
\end{itemize}

\subsection{Velocities and fluxes}\label{vel:flux}
When concentrations occur in conservation equations, leading to measure-valued solutions, a key point to 
obtain existence and uniqueness in the sense of distributions is the definition of the flux and the corresponding velocity.
This was already pointed out in \cite{BJpg}, where duality solutions are defined for pressureless gases,
and partially managed through conditions on the initial data.
A more satisfactory solution came out in \cite{jamesnv}, since uniqueness was completely handled by a careful definition of the 
flux of the equation, or in other terms, the product $\widehat{a}_\rho\rho$. An analogous situation arising in plasma physics
is considered in \cite{CRAS}, around duality solutions as well. In a similar context, other definitions of the product can
be found, see \cite{NPS} in the one-dimensional setting, and \cite{PoupaudDef} for a generalization in two space dimensions,
where defect measures are used.

We explain in more details this point in the present context, in order to give a meaning to both duality and gradient flow 
solutions in the sense of distributions. As a rule, the product of $a(W'*\rho(t))$ by $\rho(t)$ is not well-defined
when $\rho(t)\in \calM_b(\RR)$. 
First we compute $W'*\rho$. We write $\rho=-\partial_xu$, so that $u\in BV(\RR)$. For such a function, we denote
by $S_u$ the set of $x\in \RR$ where $u$ does not admit an 
approximate limit, $|S_u|=0$, and by $J_u\subset S_u$ the set of approximate jump points 
(see \cite[Proposition 3.69]{ambrosioBV}).
We use the decomposition $\rho=-\partial^a_xu+\rho^c+\rho^j$, where
$\partial^a_xu \ll {\mathcal L}$ is the regular part of the derivative, $\rho^j=\sum_{y\in J_u}\zeta_y\delta_y$ the jump part,
and $\rho^c$ the so-called Cantor part. The diffuse part of the derivative is defined as 
$\rho^d=-\partial^a_xu+\rho^c$.
For $x\notin J_u$, we easily obtain 
	$$
W'*\rho(x) = W'*\rho^d(x)+\sum_{y\in J_u}\zeta_yW'(x-y),
	$$
while if $x\in J_u$, the function is not defined. Indeed, letting $z\to x$, first with $z<x$, then with $z>x$, we obtain
	\begin{equation}\label{convol}
W'*\rho(x^{\pm}) = -W'*\partial^a_xu(x)+\sum_{y\in J_u,y\ne x}\zeta_yW'(x-y) + \zeta_xW'(0^{\pm}).
	\end{equation}
Removing the indetermination amounts to define a velocity for a single Dirac mass located in $x$ or equivalently for the
center of mass of the density. Obviously, formula \eqref{achapo} sets this value to $0$, hence a single Dirac mass is stationary,
and the product by the measure $\rho$ is meaningful. Therefore in the linear case we can consider the flux 
	$$
J(t,x) := \widehat{a}_\rho(t,x)\,\rho(t,x).
	$$
Recall that this value is obtained by computing the element of minimal norm in the subgradient of the energy $\mathcal W$ 
corresponding to $W$.

On the other  hand, in the nonlinear case, with $W''=-\delta_0+w$, the natural quantity to be defined is the flux $J$,
by formula \eqref{DefFluxJ}. 
To define the corresponding velocity, and give rigorous meaning to \eqref{ChainRule}, we use the Vol'pert calculus for
 BV functions \cite{volpert} (see also \cite[Remark 3.98]{ambrosioBV}): 
for a BV function $u$, the fonction $\achapo_u$ defining the chain rule 
$\pa_x(A(u))=\achapo_u \pa_x u$ is constructed by
	\begin{equation}\label{eq:volpert1}
\achapo_u(x) = \int_0^1 a(u_1(x)+ (1-t)u_2(x))\,dt, 
	\end{equation}
where 
	\begin{equation}\label{eq:volpert2}
(u_1,u_2)=\left\{\begin{array}{ll}
\ds (u,u) & \qquad \mbox{ if } x\in \RR\setminus S_u, \\[2mm]
\ds (u^+,u^-) & \qquad \mbox{ if } x\in J_u, \\[2mm]
\ds \mbox{ arbitrary } & \qquad \mbox{elsewhere.}
\end{array}\right.
	\end{equation}
Now we apply that to $u=W'*\rho$, and obtain, using the antiderivative $A$ of $a$, 
	\begin{equation}\label{aVolpert}
\achapo_u(x) = \left\{\begin{array}{ll}
\ds a(W'*\rho(x)) & \mbox{ if }x\in \RR\setminus S_u, \\[2mm]
\ds \frac{A(W'*\rho(x^+))-A(W'*\rho(x^-))}{W'*\rho(x^+)-W'*\rho(x^-)}& \mbox{ if } x\in J_u, \\[3mm]
\ds \mbox{ arbitrary } & \mbox{elsewhere.}
	\end{array}\right.
	\end{equation}
The connection with the linear case follows since then $A(v)=v^2/2$, hence
	$$
\dfrac{A(W'*\rho(x^+))-A(W'*\rho(x^-))}{W'*\rho(x^+)-W'*\rho(x^-)} = \dfrac{W'*\rho(x^+)+W'*\rho(x^-)}2.
	$$
Therefore the undetermined term in \eqref{convol} is replaced by $(W'(0^+)+W'(0^-))/2$, which vanishes
since $W$ is even, and we recover \eqref{achapo}.

\section{The linear case}\label{aId}
By linear case we mean the case where $a=\mbox{\rm id}$ in \eqref{EqInter}. Together with assumptions
of Definition \ref{pointyPot}, this is exactly the context of \cite{Carrillo}. At first we prove Theorem \ref{th:duality},
 obtaining existence of duality solutions in Subsection \ref{existence}, and uniqueness in Subsection \ref{uniqueness}. 
Next, in Subsection \ref{link} we establish that they are equivalent to gradient flow solutions, 
thus answering the questions raised by Remark 2.16 of \cite{Carrillo}. More precisely, we prove the following theorem.

\begin{theorem}\label{th:link}
Let $a=\mbox{id}$.
Let us assume that $W$ satisfies assumptions {\bf (A0)--(A4)} and 
that $\rho^{ini}\in \calP_2(\RR)$.
\begin{enumerate}
\item[(i)] Let $\rho$ be the duality solution as in Theorem \ref{th:duality}.
Then for all $t>0$, $\rho(t)\in \calP_2(\RR)$,
$\rho\in AC_{loc}^2((0,+\infty);\calP_2(\RR))$ and $\rho$ is the 
gradient flow solution as in Theorem \ref{GradFlow}.
\item[(ii)] If $\rho$ is the gradient flow solution of Theorem \ref{GradFlow}, 
then it is a duality solution as in Theorem \ref{th:duality}.
\end{enumerate}
\end{theorem}

\subsection{Existence of duality solutions}\label{existence}
The first step is to verify that the velocity $\achapo_\rho$ defined by \eqref{achapo}
 satisfies the OSL condition \eqref{OSLC}.
	\begin{lemma}\label{achapoOSL}
Let $\rho(t)\in \calM_b(\RR)$ be nonnegative for all $t\geq 0$.
Then under assumptions {\bf (A0) -- (A4)} the function
$(t,x)\mapsto \widehat{a}_\rho(t,x)$ defined in \eqref{achapo}
satisfies the one-sided Lipschitz estimate
	$$
\widehat{a}_\rho(t,x)-\widehat{a}_\rho(t,y) \leq \lambda (x-y)|\rho|(\RR),
\quad\mbox{ for all }\ x>y,\ t\geq 0
	$$
	\end{lemma}
\begin{proof}
By definition \eqref{achapo}, we have
	$$
\widehat{a}_\rho(x)-\widehat{a}_\rho(y) = \int_{z\neq x, z\neq y} 
(W'(x-z)-W'(y-z))\rho(dz) + W'(x-y) \int_{z\in\{x\}\cup\{y\}} \rho(dz),
	$$
where we use the oddness of $W'$ {\bf (A0)} in the last term.
Let us assume that $x>y$, from \eqref{WOSL}, we deduce that
$W'(x-z)-W'(y-z)\leq \lambda (x-y)$ and with \eqref{WOSL1}, we deduce
$W'(x-y)\leq \lambda (x-y)$. Thus, using the nonnegativity of $\rho$, 
we deduce the one-sided Lipschitz (OSL) estimate for $\widehat{a}_\rho$.
\end{proof}

\noindent
{\bf Proof of the existence result in Theorem \ref{th:duality}.} 
This proof is split in several steps.

$\bullet$ {\bf Aggregates}.

Consider first $\rho^{ini}_n=\sum_{i=1}^n m_i \delta_{x_i^0}$ 
where $x_1^0<x_2^0<\dots<x_n^0$ and the $m_i$-s are nonnegative.
We assume that $\sum_{i=1}^n m_i=1$ and that $\sum_{i=1}^n m_i |x_i^0| <+\infty$ such that
$\rho^{ini}_n\in \calP_1(\RR)$.
We look forward a solution $\rho_n(t,x)=\sum_{i=1}^n m_i \delta_{x_i(t)}$
in the distributional sense of the equation
\begin{equation}
  \label{eq:rhodistrib}
\pa_t\rho + \pa_x (\widehat{a}_\rho\rho) = 0,  
\end{equation}
where $\widehat{a}_\rho$ is defined in \eqref{achapo}.
Let $u_n:=\int^x \rho_n=\sum_{i=1}^n m_i H(x-x_i(t))$ where $H$ is the Heaviside 
function. Then we have
	$$
-\pa_t u_n = \widehat{a}_{\rho_n} \rho_n = \sum_{i=1}^n m_i \sum_{j\neq i} m_j W'(x_i-x_j) \delta_{x_i}.
	$$
In fact,
	\begin{equation}\label{achaporhon}
\widehat{a}_{\rho_n}(x) = \left\{\begin{array}{ll}
\ds \sum_{j\neq i} m_jW'(x_i-x_j) & \ds \qquad \mbox{ if } x=x_i,\ i=1,\ldots,n  \\[2mm]
\ds \sum_{j=1}^n m_j W'(x-x_j) & \ds \qquad \mbox{ otherwise}.
\end{array}\right.
	\end{equation}
From Lemma \ref{achapoOSL} and expression \eqref{achaporhon}, we deduce that 
$\widehat{a}_{\rho_n}$ satisfies the OSL condition.
Hence, there exists a unique Filippov flow \cite{Filippov} 
which is global in time thanks to assumption {\bf (A4)}.
Then the sequence $(x_i)_{i=1,\ldots,n}$ satisfies the ODE system
	\begin{equation}\label{EDOxi}
x'_i(t) = \sum_{j\neq i} m_j W'(x_i-x_j),\qquad x_i(0)=x_i^0,\qquad i=1,\ldots,n_\ell,
	\end{equation}
where $n_\ell\leq n$ is the number of distinct particles, i.e.
$n_\ell = \#\{i\in \{1,\ldots,N\},  x_i\neq x_j, \forall\,j \}$.
Equation \eqref{EDOxi} should be understood as looking for absolute continuous 
solutions to the integral problem
\beq\label{EDOint}
x_i(t) = x_i^0 + \sum_{j\neq i} \int_0^t m_j W'(x_i(s)-x_j(s))\,ds.
\eeq
Notice that \eqref{convol} rewrites as
	$$
W'*\rho_n(x_i^\pm) = \sum_{j\neq i} m_j W'(x_i-x_j) + m_i W'(0^\pm).
	$$

We define the dynamics of aggregates as follows.
\begin{itemize}
\item the $x_i$-s are solutions of system \eqref{EDOxi} (where the right-hand side is zero if $n_\ell =1$), when they are all distinct;
\item at collisions, we define a sticky dynamics: if $x_i=x_j$ 
at time $T_\ell$ when for instance $i<j$,
then the two aggregates collapse in a single one and we redefine
system \eqref{EDOxi} by changing the number $n_\ell$ to $n_\ell-1$,
replacing the mass $m_i$ by $m_i+m_j$ and deleting the point $x_j$.
We denote by $0:=T_0<T_1<\ldots T_k<\infty$ the times of collapse, where $k<n$.
\end{itemize}

This choice for the dynamics is clearly mass-conservative.
Moreover, with this choice at the collisional times, we still have
$x_i<x_j$ when $i<j$. 


Thus we can define $\rho_n:=\sum_{i=1}^{n_\ell} m_i \delta_{x_i}$. 
It is then straightforward by construction that $\rho$ is a solution 
in the distributional sense of \eqref{eq:rhodistrib}--\eqref{achapo}.

$\bullet$ {\bf Finite first order moment}.

By construction, we clearly have that $\rho\geq 0$.
Let us denote by $j_1^n(t):=\sum_{i=1}^{n_\ell} m_i |x_i(t)|$ the first order moment.
We have $j_1(0)<+\infty$.
Using \eqref{EDOxi}, we compute
$$
\frac{d}{dt} j_1^n(t) \leq \sum_{i=1}^{n_\ell}  m_i \sum_{j\neq i} m_j |W'(x_i-x_j)|
\leq C, 
$$
where we use {\bf (A4)}
and $\sum_i m_i=|\rho^{ini}_n|(\RR)=1$ for the last inequality.
We deduce that there exists a nonnegative constant $C$ such that for all $t\in [0,T]$,
\beq\label{boundj1}
j_1^n(t) \leq C T + j_1^n(0).
\eeq

$\bullet$ {\bf Duality solutions}.

From Lemma \ref{achapoOSL} and expression \eqref{achaporhon}, we deduce that 
$\widehat{a}_{\rho_n}$ satisfies the OSL condition and is piecewise continuous
outside the set of discontinuities $\{x_i\}_{i=1,\ldots,n_\ell}$. Theorem \ref{dual2distrib}
implies then that $\rho_n$ is a duality solution for all $n\in \NN^*$.

$\bullet$ {\bf Passing to the limit $n\to +\infty$}.

Using the fact that $\rho_n(t)\in \calP_1(\RR)$ for all $t\geq 0$ and from 
{\bf (A4)},
we deduce that $\widehat{a}_{\rho_n}$ is bounded in $L^\infty((0,T)\times\RR)$
uniformly with respect to $n$.
Thus, from point 4 of Theorem \ref{ExistDuality}, we can extract 
a subsequence $\widehat{a}_n$ converging in $L^\infty((0,T)\times\RR)-w\star$ towards 
$\widehat{a}$ and the corresponding sequence of duality solutions $(\rho_n)_n$ converges
in $\smes$ towards $\rho$ which is a duality solution to $\pa_t \rho + \pa_x (\widehat{a} \rho) = 0$.
Moreover, since $\rho_n\rightharpoonup \rho$ in $\smes$, the formula \eqref{achapo} defining $\widehat{a}_\rho$
implies $\widehat{a}_{\rho_n} \to \widehat{a}_\rho$ a.e.
(a proof of such result is postponed in Lemma \ref{lemA} in Appendix).
Thus $\widehat{a} = \widehat{a}_\rho$ a.e., the flux $J_n(t,x) \rightharpoonup J(t,x) := \widehat{a}_\rho \rho$ in $\calM_b(]0,T[\times \RR)-w$ and
the conservation equation \eqref{rhodis} holds both in the duality and distributional sense.
Finally passing to the limit $n\to +\infty$ in \eqref{boundj1}, we deduce that $\rho$ has a 
bounded first order moment.

\subsection{Uniqueness}\label{uniqueness}
Let $\rho$ be a nonnegative duality solution which satisfies \eqref{rhodis} in the distributional sense.
As above, we denote by $F$ the cumulative distribution function of $\rho$
and by $F^{-1}$ its generalized inverse. We have then by integration of 
\eqref{rhodis}
$$
\pa_t F + \widehat{a}_\rho \pa_x F = 0,
$$
so that the generalized inverse is a solution to
\begin{equation}
  \label{eq:F-1}
\pa_t F^{-1}(t,z) = \widehat{a}_\rho(t,F^{-1}(z)).  
\end{equation}
Moreover thanks to a change of variables in \eqref{achapo}, 
	$$
\widehat{a}_\rho(t,F^{-1}(z))=\int_{y\neq z} W'(F^{-1}(z)-F^{-1}(y))\,dy.
	$$

	\begin{proposition}\label{propuniq}
Assume $\rho_1(t,\cdot),\rho_2(t,\cdot)\in\calP_1(\RR)$ satisfy \eqref{rhodis} in the 
sense of distributions, with $\widehat{a}_{\rho_i}$ given by \eqref{achapo}, and initial data
$\rho_1^{ini}$ and $\rho_2^{ini}$. Then we have, for all $t>0$
	$$
d_{W1}(\rho_1(t,\cdot),\rho_2(t,\cdot)) \leq e^{2\lambda t} d_{W1}(\rho_1^{ini},\rho_2^{ini}).
	$$
	\end{proposition}

\begin{proof}
Let $F_i^{-1}$ denote the generalized inverse of $\rho_i$, $i=1,2$. From \eqref{eq:F-1}, we have
	$$\begin{array}{ll}
\ds \pa_t(F_1^{-1} - F_2^{-1}) 											
& \ds= \widehat{a}_{\rho_1}(t,F^{-1}_1(z))-\widehat{a}_{\rho_2}(t,F^{-1}_2(z))  \\[2mm]
&\ds = \int_{y\neq z} \big(W'(F^{-1}_1(z)-F^{-1}_1(y))
-W'(F^{-1}_2(z)-F^{-1}_2(y))\,dy.
	\end{array}$$
Multiplying the latter equation by
$sign(F_1^{-1}(z)-F_2^{-1}(z))$ 
and integrating, we get
$$
\begin{array}{l}
\ds \frac{d}{dt} \int_0^1 \big|F_1^{-1}(z)-F_2^{-1}(z)\big|\,dz = \\[3mm]
\ds \intdouble_{\{y\neq z\}}\big(W'(F^{-1}_1(z)-F^{-1}_1(y))
-W'(F^{-1}_2(z)-F^{-1}_2(y))\big)
\mbox{sign}\big(F_1^{-1}(z)-F_2^{-1}(z)\big)\,dzdy.
\end{array}
$$
Using the oddness of $W'$ and exchanging the role of $y$ and $z$ in the 
integral above, we also have
$$
\begin{array}{l}
\ds \frac{d}{dt} \int_0^1 \big|F_1^{-1}(z)-F_2^{-1}(z)\big|\,dz = \\[3mm]
-\ds \intdouble_{\{y\neq z\}}\big(W'(F^{-1}_1(z)-F^{-1}_1(y))
-W'(F^{-1}_2(z)-F^{-1}_2(y))\big)
\mbox{sign}\big(F_1^{-1}(y)-F_2^{-1}(y)\big)\,dzdy.
\end{array}
$$
Then we deduce
	\begin{equation}\label{uniq1}\begin{array}{ll}
\ds \frac{d}{dt} \int_0^1 \big|F_1^{-1}(z)-F_2^{-1}(z)\big|\,dz = &\ds\frac 12
\ds \intdouble_{\{y\neq z\}}\big(W'(F^{-1}_1(z)-F^{-1}_1(y))
-W'(F^{-1}_2(z)-F^{-1}_2(y))\big)\times  \\[4mm]
&\ds \big(\mbox{sign}\big(F_1^{-1}(z)-F_2^{-1}(z)\big)-\mbox{sign}\big(F_1^{-1}(y)-F_2^{-1}(y)\big)\big)\,dzdy.
	\end{array}\end{equation}
The one-sided Lipschitz estimate for $W'$ in \eqref{WOSL}
implies that the integrand in the right-hand side is bounded by
	$$
2\lambda\big|F_1^{-1}(z)-F_1^{-1}(y)-F_2^{-1}(z)+F_2^{-1}(y)\big|.
	$$
Hence, after an integration, we deduce
	$$
\frac{d}{dt}\int_0^1 |F_1^{-1}-F_2^{-1}|(z)\,dz \leq 2\lambda \int_0^1 |F_1^{-1}-F_2^{-1}|(z)\,dz.
	$$
Since $\|(F_1^{-1}-F_2^{-1})(t)\|_{L^1(0,1)}= d_{W1}(\rho_1,\rho_2)$, we conclude the proof by a Gronwall argument.
\end{proof}

{\bf Proof of Theorem \ref{th:duality}.}
The existence has been obtained in Section \ref{existence}.
Then if we have two duality solutions $\rho_1$ and $\rho_2$ as 
in Theorem \ref{th:duality}, Proposition \ref{propuniq} implies 
that their generalized inverse are equal. Therefore $\rho_1=\rho_2$.
Finally, the second point of Theorem \ref{ExistDuality} allows to
define the duality solution as the push-forward of $\rho^{ini}$ by the backward flow.
\qed

\subsection{Proof of Theorem \ref{th:link}}\label{link}
To cope with gradient flow solutions, we need first to prove that the second order moment 
is bounded provided $\rho^{ini}\in \calP_2(\RR)$.
We follow the idea of the proof of finite first order moment
in subsection \ref{existence}: we consider an approximation
of $\rho^{ini}$ by $\sum_{i=1}^n m_i \delta_{x_i^0}$ and build
the corresponding duality solution $\rho_n(t,x)=\sum_{i=1}^{n_\ell} m_i \delta_{x_i(t)}$
where the dynamics of the nodes $\{x_i\}_{i=1,\ldots,n_\ell}$ is given in \eqref{EDOxi}.
Let us denote by $j_2^n(t):=\sum_{i=1}^{n_\ell} m_i x_i^2(t)$ the second order moment.
We have $j_2^n(0)<+\infty$. Using \eqref{EDOxi}, we compute
$$
\frac{d}{dt}j_2^n(t) = \sum_{i=1}^{n_\ell} 2 m_i x_i \sum_{j\neq i} m_j W'(x_i-x_j)
\leq C \sum_{i=1}^{n_\ell} m_i |x_i| = C j_1^n(t).
$$
where we use {\bf (A4)} 
for the last inequality.
Since $j_1^n$ is uniformly bounded (see \eqref{boundj1}), we deduce that $j_2^n$
is uniformly bounded on $[0,T]$ by a constant only depending on $T$ and on $j_2^n(0)$.
Then we can pass to the limit $n\to +\infty$ to obtain a bound 
on $\int_\RR |x|^2 \rho(t,dx)$ for any $t>0$.
Moreover, 
we deduce that the velocity field $\widehat{a}_\rho$ defined in \eqref{achapo} 
is bounded in $L^1((0,T);L^2(\rho))$.

Now, if $\rho$ is a duality solution as in Theorem \ref{th:duality}, then
it satisfies \eqref{rhodis} in the distributional sense.
Using \cite[Theorem 8.3.1]{Ambrosio} and the $L^1((0,T);L^2(\rho))$
bound on the velocity $\widehat{a}_\rho$, we deduce that 
$\rho\in AC_{loc}^2((0,+\infty);\calP_2(\RR))$.
Thus $\rho$ is a gradient flow solution (see \cite{Carrillo}
or \cite[Sections 8.3 and 8.4]{Ambrosio}). This concludes the
proof of point $(i)$ of Theorem \ref{th:link}.

Conversely, if $\rho$ is a gradient flow solution of Theorem \ref{GradFlow},
by uniqueness of both duality solutions and gradient flow solutions,
we deduce that $\rho$ is also a duality solution.
\qed

\Section{On the case $a\neq \mbox{id}$}\label{anonid}

The situation $a\neq\mbox{id}$ is not so favourable as the previous one, and
one has to impose restrictions on the potential $W$. First we recall that attractivity implies that $a$ is
non-decreasing, see \eqref{hyp_a}. Next, we need to assume that $W$ has the following structure, see \eqref{hyp},
	$$
W'' = -\delta_0 + w,\qquad w \in Lip\cap L^\infty(\RR).
	$$
With these assumptions, we are able to prove existence and uniqueness of duality solutions, Theorem \ref{dual_anonid}, this is 
the aim of subsection \ref{dualanonid}. Next, in subsection \ref{gradflownonid}, we turn to gradient flow, which are definitely
 not well suited for that case, since we have to restrict ourselves to $w=0$ in the previous assumption on $W$.

\subsection{Duality solutions}\label{dualanonid}
Here we prove Theorem \ref{dual_anonid}, following the same strategy as in the linear case: first we prove the OSL condition, 
next establish the dynamics of aggregates, which leads to existence by approximation. Finally, uniqueness follows from a 
contraction principle in the space ${\mathcal P}_1$. In addition, we prove that duality solutions are absolutely continuous 
in time.

\subsubsection{OSL condition}
The first step consists in checking the OSL property for $a$.
\begin{lemma}\label{aOSL}
Assume $0\leq\rho\in {\cal M}_b(\RR)$ and that \eqref{hyp_a} holds, 
If Assumption \ref{assump} is satisfied, then the function
$x\mapsto a(W'*\rho)$ satisfies the OSL condition \eqref{OSLC}.
\end{lemma}
\begin{proof}
Using \eqref{hyp}, we deduce that
$$
\pa_{xx}W*\rho = -\rho + w*\rho.
$$
Therefore,
$$
\pa_x(a(\pa_xW*\rho)) = a'(\pa_xW*\rho)(-\rho+w*\rho)
\leq a'(\pa_xW*\rho) w*\rho,
$$
where we use the nonnegativity of $\rho$ in the last inequality.
Then from \eqref{hyp_a} we get
$$
\pa_x(a(\pa_xW*\rho)) \leq \alpha \|\rho\|_{L^1} \|w\|_{L^\infty}.
$$
It implies the OSL condition on the velocity.
\end{proof}

\subsubsection{Proof of the existence result in Theorem \ref{dual_anonid}}\label{exist_anonid}

{\bf Approximation by aggregates.} 

Following the idea in subsection \ref{existence}, we first approximate 
the initial data $\rho^{ini}$ by a finite sum of Dirac masses:
$\rho^{ini}_n=\sum_{i=1}^n m_i \delta_{x_i^0}$ where $x_1^0<x_2^0<\dots<x_n^0$
and the $m_i$-s are nonnegative.
We assume that $\sum_{i=1}^n m_i=1$ and $\sum_{i=1}^n m_i |x_i^0|<+\infty$,
i.e. $\rho^{ini}_n\in \calP_1(\RR)$.
We look for a sequence $(\rho_n)_n$ solving in the distributional sense
$\pa_t\rho_n +\pa_x J_n=0$ where the flux $J_n$ is given by \eqref{DefFluxJ}.
A function $\rho_n(t,x)=\sum_{i=1}^n m_i \delta_{x_i(t)}$ is such a 
solution provided the function $u_n$ defined by
\beq\label{defu}
u_n(t,x):=\int^x\rho_n\,dx = \sum_{i=1}^n m_i H(x-x_i(t)),
\eeq
where $H$ denotes the Heaviside function, is a distributional solution to
\beq\label{equdis}
\pa_t u_n -\pa_x\big(A(\pa_xW*\rho_n)\big) + a(\pa_xW*\rho_n) w*\rho_n =0.
\eeq
From \eqref{hyp}, we deduce that
\beq\label{derW}
W'(x) = - H(x) + \widetilde{w}(x), \quad \mbox{ where }
\widetilde{w}(x) = \int^x_0 w(y)\,dy +\frac 12.
\eeq
Then, we have
\begin{equation}
  \label{eq:Wrho+}
W'*\rho_n(x_i^+) = -\sum_{j=1}^i m_j + 
\sum_{j=1}^n m_j \widetilde{w}(x_i-x_j).
\end{equation}
And
\begin{equation}
  \label{eq:Wrho-}
W'*\rho_n(x_i^-) = m_i + W'*\rho_n(x_i^+).
\end{equation}
From these identities together with \eqref{aVolpert}, 
straightforward computations show that in the distributional sense
\beq\label{calAdis}
\pa_x \big(A(W'*\rho_n)\big) = a(W'*\rho_n)w*\rho_n + 
\sum_{i=1}^n [A(W'*\rho_n)]_{x_i} \delta_{x_i},
\eeq
where $[f]_{x_i}=f(x_i^+)-f(x_i^-)$ is the jump of the function $f$ at $x_i$.
Injecting \eqref{defu} and \eqref{calAdis} in \eqref{equdis}, we find
$$
-\sum_{i=1}^n m_i x'_i(t) \delta_{x_i(t)} = \sum_{i=1}^n [A(W'*\rho_n)]_{x_i} \delta_{x_i}.
$$
Thus it is a solution if we have
\beq\label{dynagg}
m_i x'_i(t) = -[A(W'*\rho_n)]_{x_i(t)}, \quad \mbox{ for } i=1,\dots, n.
\eeq
This system of ODEs is complemented by the initial data $x_i(0)=x_i^0$.
Thus we are looking for absolute continuous solutions to the integral problem
\begin{equation}\label{xibis}
x_i(t) = x_i^0 + \int_0^t \frac{[A(W'*\rho_n(s))]_{x_i(s)}}{[W'*\rho_n(s)]_{x_i(s)}}\,ds ,
\qquad i=1,\ldots n.
\end{equation}

Then we define the dynamics of aggregates as in subsection \ref{existence}~:
\begin{itemize}
\item When the $x_i$ are all distinct, they are solutions of system 
\eqref{dynagg} or equivalently \eqref{xibis} (with zero right hand 
side if $n_\ell =1$, where we recall $n_\ell(t)$ is the number of
distinct particles at time $t$).
\item At collisions, we use we use the same sticky dynamics as above.
\end{itemize}
We recall that this choice of the dynamics implies mass conservation. As above, we have existence of the sequence $(x_i)_i$ satisfying 
\eqref{xibis} on $[0,T]$ with initial condition $(x_i^0)$.
Then we set $\rho_n(t,x) = \sum_{i=1}^{n_\ell} m_i\delta_{x_i(t)}(x)$. 
By construction, $\rho_n$ is a solution in the sense of distribution of
\eqref{eqrhodis}-\eqref{DefFluxJ} for given initial data $\rho_n^{ini}$.

{\bf Finite first order moment.}

As in subsection \ref{existence}, we define $j_1^n(t)=\sum_{i=1}^{n_\ell} m_i |x_i(t)|$
and we compute
	$$
\frac{d}{dt}j_1^n(t)= \sum_{i=1}^{n_\ell} m_i \frac{x_i}{|x_i|}\frac{[A(W'*\rho_n)]_{x_i}}{[W'*\rho_n]_{x_i}},
	$$
where we use \eqref{xibis}. 
From {\bf (A4)} and the fact that $a'$ is bounded \eqref{hyp_a}, we deduce that
$a(W'*\rho_n)$ is uniformly bounded.
Moreover, since $a$ is nondecreasing, $A$ is a convex function, therefore the quantity
$\ds \frac{[A(W'*\rho_n)]_{x_i}}{[W'*\rho_n]_{x_i}}$ is uniformly bounded.
Then we have
\beq\label{boundj1bis}
j_1^n(t) \leq C T + j_1^n(0), \qquad \forall t\in [0,T],
\eeq
where $C$ stands for a generic nonnegative constant.

{\bf Existence of duality solutions}


By the Vol'pert calculus recalled in Section \ref{vel:flux}, we have
$$
J_n:=-\pa_x(A(W'*\rho_n))+a(W'*\rho_n)w*\rho_n = \achapo_n \rho_n, 
\quad \mbox{ and } \quad \achapo_n=a(W'*\rho_n) \mbox{ a.e. }
$$
Then $\rho_n$ is a solution in the distributional sense of
$$
\pa_t\rho_n + \pa_x(\achapo_n \rho_n) = 0.
$$
Moreover, by definition $a(W'*\rho_n)$ is piecewise
continuous with the discontinuity lines defined by $x=x_i$, $i=1,\dots,n$,
and by assumption \ref{assump1} it is bounded in $L^\infty$.
We can apply Theorem \ref{dual2distrib} which gives that 
$\rho_n$ is a duality solution and that $\achapo_n$ is a universal
representative of $a(W'*\rho_n)$. Then the flux is given by
$a(W'*\rho_n)\Dpetit \rho_n=J_n$.

{\bf General case. }

Let us yet consider the case of any initial data $\rho^{ini}\in \calP_1(\RR)$.
We approximate $\rho^{ini}$ by 
$\rho^{ini}_n=\sum_{i=1}^n m_i \delta_{x_i^0}$, $\rho^{ini}_n\in \calP_1(\RR)$
with $\rho_n^{ini}\rightharpoonup \rho^{ini}$ in $\calM_b(\RR)$.
By the same token as above, we can construct a sequence of solutions 
$(\rho_n)_n$ with $\rho_n(t=0)=\rho_n^{ini}=\sum_{i=1}^n m_i\delta_{x_i^0}$, 
which solves in the sense of distributions
$$
\pa_t\rho_n + \pa_x J_n=0, 
\quad J_n = -\pa_x \big(A(\pa_xW*\rho_n)\big)+a(\pa_xW*\rho_n) w*\rho_n,
$$
and which satisfies
$$
\achapo_n \rho_n = J_n, \quad \achapo_n=a(W'*\rho_n) \mbox{ a.e. }
$$
Moreover, since $W'*\rho_n$ is bounded in $L^\infty$ uniformly with 
respect to $n$ by construction and assumption {\bf (A4)},
we can extract a subsequence of $(a(W'*\rho_n))_n$ that converges
in $L^\infty -weak*$ towards $b$.
Since from Lemma \ref{aOSL}, $a(W'*\rho_n)$ satisfies the OSL condition, 
we deduce from Theorem \ref{ExistDuality} 4) that, up to an extraction,
$\rho_n\rightharpoonup \rho$ in $\smes$ and 
$\achapo_n\rho_n\rightharpoonup \achapo \rho$
weakly in $\calM_b(]0,T[\times\RR)$, $\rho$ being a duality solution
of the scalar conservation law with coefficient $b$. Then 
$J_n\to J:=-\pa_x(A(W'*\rho))+a(W'*\rho)w*\rho$ in ${\mathcal D}'(\RR)$ 
and that $a(W'*\rho_n)\to a(W'*\rho)$ a.e. By uniqueness of the weak
limit, we have $b=a(W'*\rho)$. Moreover $J=\achapo \rho$ a.e. and
$\rho$ satisfies then \eqref{eqrhodis}.
Finally, we recover the bound on the first order moment by passing to the limit 
$n\to +\infty$ in the estimate \eqref{boundj1bis}.
\qed

\begin{remark}
Let us consider the case studied in the previous Section~: $a=\mbox{id}$ and
$W$ is even. Since $W'$ is odd, then \eqref{derW} rewrites
	$$
W'(x) = -H(x) + w_0(x), \qquad \mbox{where }\ w_0(x)=\int_0^x w(y)\,dy+\frac 12.
	$$
When $a=\mbox{id}$, we have $A(x)=\frac 12 x^2$. Then system \eqref{dynagg}
rewrites
	$$
m_i x_i'(t) = -\frac 12 (W'*\rho_n(x_i^+)-W'*\rho_n(x_i^-))
(W'*\rho_n(x_i^+)+W'*\rho_n(x_i^-)).
	$$
Then, from \eqref{eq:Wrho+} and \eqref{eq:Wrho-}, we have
	$$
x'_i(t)= -\sum_{j=1}^{i-1} m_j -\frac{m_i}{2} +
\sum_{j=1}^n m_j w_0(x_i-x_j)= -\sum_{j=1}^{i-1} m_j +\sum_{j\neq i} m_j w_0(x_i-x_j).
	$$
From the expression of $W'$ above, we deduce that
$x'_i(t)=\sum_{j\neq i} m_j W'(x_i-x_j)$ and we recover 
the dynamical system \eqref{EDOxi} of the previous Section.
\end{remark}

\begin{remark}
The dynamical system \eqref{dynagg} defines actually the macroscopic velocity.
Indeed, if we formally take the limit $n\to+\infty$ of the 
right-hand side of \eqref{xibis}, this latter term converges towards the velocity $\achapo_u$
defined by the chain rule \eqref{aVolpert}
\end{remark}

\subsubsection{Uniqueness.}

We first notice that the strategy used in subsection \ref{uniqueness} cannot 
be used here, since it strongly relies on the linearity of $a$.
Then we have to use the approach proposed in \cite{jamesnv} which uses
an entropy estimate.
The key point is to observe that the quantity $W'*\rho$ solves a
scalar conservation laws with source term.
\begin{proposition}[Entropy estimate]\label{entropy}
Let us assume that Assumptions \ref{assump} and \eqref{hyp_a} hold.
For $T>0$, let $\rho\in C([0,T],\calP_1(\RR))$ satisfying in the distribution sense
\eqref{eqrhodis}-\eqref{DefFluxJ}.
Then $u:=W'*\rho$ is a weak solution of 
\beq\label{eq:U}
\pa_t u+\pa_x A(u) = a(u)w*\rho+\pa_x(w*A(u))-w*(a(u)w*\rho).
\eeq
Moreover, if we assume that the entropy condition
\beq\label{entropycond}
\pa_x u \leq w*\rho
\eeq
holds, then for any twice continuously differentiable convex function $\eta$,
we have
\beq\label{eq:etaU}
\pa_t \eta(u)+\pa_x q(u) -\eta'(u) a(u)w*\rho+\eta'(u)\big(\pa_x(w*A(u))-w*(a(u)w*\rho))
\leq 0,
\eeq
where the entropy flux is given by $\ds q(x)=\int_0^x \eta'(y)a(y)\,dy$.
\end{proposition}
\begin{proof}
Equation \eqref{eq:U} is obtained by taking the convolution product
of \eqref{eqrhodis} with $W'$.
The entropy inequality is then a straightforward adaptation of the proof of Lemma
4.5 of \cite{jamesnv}.
\end{proof}

We turn now to the proof of the uniqueness.
Once again, we use the idea developed in \cite{jamesnv} 
and extend it to the case at hand.
Consider two solutions $\rho_1$ and $\rho_2$ such as in 
Theorem \ref{dual_anonid}. We denote $u_1:=W'*\rho_1$ and 
$u_2=W'*\rho_2$.
Starting from the entropy inequality \eqref{eq:etaU} with the family 
of Kru\v{z}kov entropies $\eta_\kappa(u) = |u-\kappa|$ and using the
doubling of variable technique developed by Kru\v{z}kov, 
we obtain as in the proof of Theorem 5.1 of \cite{jamesnv}
$$
\frac{d}{dt} \int_\RR \big|u_1-u_2\big| \leq 
\|w\|_{Lip} \int_\RR \big|A(u_1)-A(u_2)\big|\,dx + 
\big(1+\|w\|_\infty\big) \int_\RR \big|a(u_1)w*\rho_1-a(u_2)w*\rho_2\big|\,dx.
$$
From {\bf (A4)} 
and the bound of $\rho(t)$ in $\calP_1(\RR)$ for all $t$, 
we deduce that $u_i$, $i=1,2$ are bounded in $L^\infty_{t,x}$.
Then we get
\beq\label{uniqU}
\frac{d}{dt} \int_\RR \big|u_1-u_2\big| \leq C\Big(
\int_\RR \big|u_1-u_2\big|\,dx + \int_\RR \big|w*\rho_1-w*\rho_2\big|\,dx\Big),
\eeq
where we use moreover \eqref{hyp_a}.
Taking the convolution with $w$ of equation \eqref{eqrhodis} we deduce
$$
\pa_t w*\rho_i - \pa_x\big(w*A(u_i)\big) + w*\big(a(u_i)w*\rho_i\big)=0, \qquad i=1,2.
$$
We deduce from \eqref{hyp_a} and the Lipschitz bound of $w$ that
\beq\label{uniqw}
\frac{d}{dt} \int_\RR \big|w*\rho_1-w*\rho_2\big| \leq C\Big(
\int_\RR \big|u_1-u_2\big|\,dx + \int_\RR \big|w*\rho_1-w*\rho_2\big|\,dx\Big).
\eeq
Adding \eqref{uniqU} and \eqref{uniqw}, we deduce applying a Gronwall lemma that
$u_1=u_2$ and $w*\rho_1=w*\rho_2$, which implies $\rho_1=\rho_2$. \qed

\begin{remark}
We point out that the entropy condition \eqref{entropycond} 
is equivalent to $\rho\geq 0$, which is required since $\rho$ is supposed to
be a density. As a consequence, if we allow $\rho$ to be nonpositive, 
uniqueness of solutions is not guaranteed (see Section 5.3 of \cite{jamesnv}
for a counter-example of uniqueness in the case $W(x)=\frac 12 e^{-|x|}$).
\end{remark}

\subsubsection{Absolute continuity.}
Following subsection \ref{link}, we have the following result:
	\begin{proposition}\label{AC2}
Under the assumptions of Theorem \ref{dual_anonid}, if moreover
$\rho^{ini}\in \calP_2(\RR)$, then the duality solution of Theorem
\ref{dual_anonid} satisfies
$\rho\in AC_{loc}^2((0,+\infty);\calP_2(\RR))$.
	\end{proposition}
\begin{proof}
The proof of the finite second order moment follows straightforwardly
the one for the first order moment in subsection \ref{exist_anonid}.
Then from \eqref{hyp_a} and {\bf (A4)},
we have that $a(W'*\rho)$ is uniformly bounded in $L^\infty$.
Therefore the velocity field $\widehat{a}$ is bounded in $L^1((0,T);L^2(\rho))$.
Moreover $\rho$ is a solution in the distributional sense of 
$\pa_t\rho + \pa_x(\widehat{a}\rho) = 0$.
We conclude then from Theorem 8.3.1 of \cite{Ambrosio} that
$\rho\in AC_{loc}^2((0,+\infty);\calP_2(\RR))$.
\end{proof}

\subsection{Gradient flows}\label{gradflownonid}
When $a\neq \mbox{id}$, equation \eqref{EqInter} does not have a structure of
a gradient. Moreover, we are not able to determine a conserved energy 
corresponding to the system. Nevertheless, in the particular case $W(x)=-\frac 12 |x|$, we are able to
adapt the technique of \cite{Carrillo} to recover the existence 
of gradient flow solutions.
Before inroducing the energy functional $\calW$ for this case, let us
first recall the observation of Section \ref{vel:flux}. 
Denoting $A$ an antiderivative of $a$, we have from the 
Vol'pert calculus
\beq\label{def:u}
\pa_x (A(u)) = \achapo_u(x)\pa_xu,
\quad \mbox{ where }\ 
u(x) = \int_{y\neq x}W'(x-y)\rho(dy)=\frac 12\big(\rho(x,+\infty)-\rho(-\infty,x)\big).
\eeq
The function $\achapo_u$ is defined in \eqref{aVolpert} and we recall 
that when $a=id$, $\achapo_u=u$.
Then, we define for $\rho\in \calP_2(\RR)$ the functional
\beq\label{Wanonid}
\calW(\rho) = -\int_\RR x \achapo_u(x) \rho(dx), \qquad u=W'*\rho.
\eeq
We first verify  that when $a=\mbox{id}$ this energy is equal to 
the one introduced in \eqref{nrjinter}. In fact, we have for 
$W(x)=-\frac 12 |x|$,
\begin{equation}\label{egalW1}
\frac 14 \int_{\RR^2} |x-y|\rho(dx)\rho(dy) = 
\frac 14 \int_\RR\int_{-\infty}^x (x-y)\rho(dy)\rho(dx) -
\frac 14 \int_\RR\int_x^{+\infty} (x-y)\rho(dy)\rho(dx).
\end{equation}
Repeated use of Fubini's Theorem in the last term of the right hand side leads to
	$$\begin{array}{ll}
\ds \frac 14 \int_{\RR^2} |x-y|\rho(dx)\rho(dy) &\ds = 
\frac 12 \int_\RR\int_{-\infty}^x (x-y)\rho(dy)\rho(dx)  \\[3mm]
&\ds = \frac 12 \int_\RR x\int_{-\infty}^x \rho(dy)\rho(dx)
-\frac 12 \int_\RR\int_y^{+\infty}\rho(dx) y\rho(dy).
	\end{array}$$
With the definition of $u$ in \eqref{def:u}, we have
	$$
\frac 14 \int_{\RR^2} |x-y|\rho(dy)\rho(dx) = -\int_\RR x u(x) \rho(dx),
	$$
which concludes the proof. 

We are able to prove in this case the following Theorem.
	\begin{theorem}\label{th:GFanonid}
Let $W(x)=-\frac 12 |x|$ and $a$ satisfy assumption \eqref{hyp_a}.
Let $\rho^{ini}\in \calP_2(\RR)$ be given.
\begin{itemize}
\item[(i)] There exists a unique gradient flow 
solution $\rho\in AC_{loc}^2([0,+\infty),\calP_2(\RR))$ in the sense of
Definition \ref{defgradflow}. Therefore $\rho$ satisfies in the 
distributional sense 
	$$
\pa_t\rho + \pa_x(\achapo_u\rho) = 0, \quad \mbox{ with }
\quad \rho(0)=\rho^{ini},
	$$
where $u(x)=W'*\rho(x)=\frac 12(\rho(x,+\infty) -\rho(-\infty,x))$ and 
$\achapo_u$ is defined in \eqref{aVolpert}.
Moreover, this solution is unique and 
we have the energy estimate~: for all $0\leq t_0\leq t_1 <\infty$,
	$$
\int_{t_0}^{t_1} |\achapo_u(t,x))|^2\,\rho(t,dx)\,dt  + \calW(\rho(t_1))
	= \calW(\rho(t_0)).
	$$
\item[(ii)] The duality solution of Theorem \ref{dual_anonid} satisfies
$\rho(t)\in \calP_2(\RR)$ for all $t\geq 0$ and coincides with the gradient flow
solution of the first item. Moreover we have $\rho=-\pa_xu$ where $u$
is the unique entropy solution of the scalar conservation law 
$$
\pa_t u + \pa_x A(u) = 0, \qquad u(0,x)=W'*\rho^{ini},
$$
where $A$ is an antiderivative of $a$.
\end{itemize}
\end{theorem}
Notice that the equivalence between entropy solutions and gradient flow solutions 
(item $(ii)$) has been observed independently in \cite{bonaschi} 
in the linear case $a=Id$ and for $W(x)=\pm |x|$ (including then a repulsive case).

\begin{proof}
$(i)$ We use the ideas of Section 2 of \cite{Carrillo}
recalled in the beginning of this paper, Section \ref{sec:Carrillo}.
The proof is divided into several steps.

We first notice that due to the one dimensional framework, we can simplify the computations
by working in the Hilbert space $L^2(-\frac 12,\frac 12)$.
In fact, the function $u=W'*\rho$ is, up to a constant, the cumulative distribution of $-\rho$,
since $\rho=-\pa_x u$.
Using the definition \eqref{aVolpert}, since $\achapo_u$ is arbitrary on $S_u\setminus J_u$,
we have~:
$$
\int_\RR x \achapo_u \rho(dx) = \int_{\RR\setminus J_u} x a(u(x)) \rho(dx) + 
\sum_{x\in J_u} x \frac{A(u(x^+))-A(u(x^-))}{u(x^+)-u(x^-)} (u(x^+)-u(x^-)).
$$
We introduce the generalized inverse $v$ of $u$ whose definition is
$$
v(t,z)=u^{-1}(t,z) := \inf \{ x\in \RR / u(t,x)>z\}.
$$
Since $u$ is nonincreasing, $v$ is nonincreasing.
For $x\in J_u$, we denote $z_x^-=u(x^-)$ and $z_x^+=u(x^+)$.
We deduce after a change of variable that
$$
\calW(\rho)= - \int_{(-\frac 12,\frac 12)\setminus u(J_u)} v(t,z) a(z) dz
- \sum_{x\in J_u} v(t,z) (A(z_x^+)-A(z_x^-)).
$$
By definition, the function $v$ is constant on the set $u(J_u)$. We deduce,
recalling that $A$ is an antiderivative of $a$,
$$
\sum_{x\in J_u} v(t,z) (A(z_x^+)-A(z_x^-)) = \sum_{x\in J_u}
\int_{z_x^-}^{z_x^+} v(t,z) a(z) dz.
$$
Thus, we can rewrite the functional $\calW$ as
\beq\label{eq:WL2}
\calW(\rho) = \widetilde{\calW}(v) :=-\int_{-1/2}^{1/2} v(t,z) a(z) \,dz.
\eeq
We define moreover $v^{ini}= (u^{ini})^{-1}$, where 
$u^{ini}=\frac 12\big(\rho^{ini}(x,+\infty)-\rho^{ini}(-\infty,x)\big)$.

$\bullet$ {\bf $-\achapo_u$ is the unique element of minimal $L^2(\rho)$-norm 
in the subdifferential $\pa\calW$ of $\calW$.}

Let $\rho\in \calP_2(\RR)$ and $u=\frac 12(\rho(x,+\infty) -\rho(-\infty,x))$.
As above, we denote $v$ the pseudo-inverse of $u$.
We first show that $-\achapo_u\in \pa\calW(\rho)$,
where $\achapo_u$ is defined in \eqref{aVolpert}. 
From \cite[Definition 10.3.1]{Ambrosio} (see also equation (10.3.12) of the same book), 
it means that for all $\mu$ in $\calP_2(\RR)$, we have
\beq\label{adW}
\calW(\mu)-\calW(\rho) \geq 
\inf_{\gamma\in \Gamma_0(\rho,\mu)}\int_\RR -\achapo_u(x)(y-x)\,\gamma(dx,dy) + o(d_{W2}(\mu,\rho)).
\eeq
For $\mu\in \calP_2(\RR)$ we denote $u_\mu=W'*\mu$ and $v_\mu$ its pseudo-inverse.
As for \eqref{eq:WL2} with $\mu$ instead of $\rho$, we deduce~:
\beq\label{eq:Wmurho1}
\calW(\mu)-\calW(\rho) = - \int_{-1/2}^{1/2} a(z) \big(v_\mu(z)-v(z)\big)\,dz.
\eeq

We have recalled in Subsection \ref{sec:1D} that in one dimension, the set of optimal 
map is given by $\Gamma_0(\rho,\mu)=\{(v,v_\mu)_\#{\LL}_{(-1/2,1/2)} \}$.
Therefore we have that for $\gamma \in \Gamma_0(\rho,\mu)$,
$$
\int_\RR -\achapo_u(x)(y-x)\,\gamma(dx,dy) = - \int_{-1/2}^{1/2} \achapo_u(v(z)) \big(v_\mu(z)-v(z)\big)\,dz.
$$
Let us consider then the quantity 
$$
R := \int_{-1/2}^{1/2} \big(\achapo_u(v(z))-a(z)\big) \big(v_\mu(z)-v(z)\big)\,dz.
$$
We have from \eqref{eq:Wmurho1}
\beq\label{eq:Wmurho2}
\calW(\mu)-\calW(\rho) = - \int_{-1/2}^{1/2} \achapo_u(v(z)) \big(v_\mu(z)-v(z)\big)\,dz
+ R.
\eeq

Using definition \eqref{aVolpert}, we have that on $(-\frac 12,\frac 12)\setminus u(J_u)$, 
$\achapo_u(v(z))=a(z)$. 
For $x\in J_u$, we denote $z_x^-=u(x^-)$ and $z_x^+=u(x^+)$, we have 
$$
\forall \,z\in (z_x^-,z_x^+), \quad
\achapo_u(v(z))= \frac{A(z_x^+)-A(z_x^-)}{z_x^+-z_x^-}
=\frac{1}{z_x^+-z_x^-} \int_{z_x^-}^{z_x^+} a(y)\,dy.
$$
Therefore, we can rewrite
$$
R = \sum_{x\in J_u} \frac{1}{z_x^+-z_x^-} \int_{z_x^-}^{z_x^+}\int_{z_x^-}^{z_x^+} \big(a(y)-a(z)\big)
\big(v_\mu(z)-v(z)\big)\,dydz.
$$
Hence, if $v_\mu-v$ is piecewise constant on $(z_x^-,z_x^+)$ for each $x\in J_u$, we have that
$R=0$. By definition, for each $x\in J_u$, we have that $v$ is constant on $(z_x^-,z_x^+)$.
Let $(v_\mu^n)_{n\in \NN}$ be a sequence of approximation of $v_\mu$ such that $v_\mu^n$ 
is piecewise constant on $(z_x^-,z_x^+)$ for all $x\in J_u$ and converge in $L^2(-\frac 12,\frac 12)$ 
and a.e. towards $v_\mu$.
Moreover $v^n_\mu$ is the monotone rearrangement of $\mu_n= -\pa_x u^n_\mu$ where
$u^n_\mu(x)=\sup\{z\in (-\frac 12,\frac12) / v^n_\mu(z)< x\}$.
We deduce from the above discussion and from \eqref{eq:Wmurho2} that
$$
\calW(\mu_n)-\calW(\rho) = - \int_{-1/2}^{1/2} \achapo_u(v(z)) \big(v_\mu^n(z)-v(z)\big)\,dz.
$$
Using the Fatou Lemma, we conclude by letting $n\to +\infty$
$$
\calW(\mu)-\calW(\rho) \geq - \int_{-1/2}^{1/2} \achapo_u(v(z)) \big(v_\mu(z)-v(z)\big)\,dz
$$

\medskip

To prove that $-\achapo_u$ is an element of minimal norm 
in $\pa\calW(\rho)$, we consider $\xi\in C^\infty\cap Lip(\RR)$
and for $\eps>0$ small enough such that $(id+\eps\xi)$ is increasing.
We have
$$
\calW((id+\eps\xi)_\# \rho)= -\int_\RR (id+\eps\xi)(x) \achapo_{u_\eps}(x+\eps\xi(x))\rho(dx),
$$
where $u_\eps=W'*((id+\eps\xi)_\# \rho)$. 
We notice that for any increasing function $\theta$, we have
$$
\begin{array}{ll}
\ds W'*(\theta_\#\rho)(\theta(x))
& \ds = \frac 12\big(\rho\big(\theta^{-1}(\theta(x),+\infty)\big)-\rho\big(\theta^{-1}(-\infty,\theta(x))\big)\big)  \\[2mm]
& \ds = \frac 12\big(\rho\big(\{ y\in\RR / \theta(y)>\theta(x)\}\big) -\rho\big(\{ y\in\RR / \theta(y)<\theta(x)\}\big)\big).
\end{array}
$$
Due to the monotonicity of the function $\theta$, we have
$\{ y\in\RR / \theta(y)<\theta(x)\} = \{y\leq x\}$.
Hence $W'*(\theta_\#\rho)(\theta(x))=W'*\rho(x)$. We deduce
$$
\calW((id+\eps\xi)_\# \rho)= -\int_\RR (id+\eps\xi)(x) \achapo_u(x)\rho(dx).
$$
Thus,
$$
\lim_{\eps\to 0}\frac{\calW((id+\eps\xi)_\#\rho)-\calW(\rho)}{\eps} =
-\int_\RR \achapo_u(x)\xi(x)\,d\rho(x).
$$
Then, from the definition of the slope \eqref{slope}, we have
$$
\liminf_{\eps\searrow 0} \frac{\calW((id+\eps\xi)_\#\rho)-\calW(\rho)}{d_{W2}((id+\eps\xi)_\#\rho,\rho)} 
\geq - |\pa \calW|(\rho).
$$
We deduce from above that
$$
\int_\RR \achapo_u(x)\xi(x)\,d\rho(x) \leq |\pa \calW|(\rho)  \liminf_{\eps\searrow 0}
\frac{d_{W2}((id+\eps\xi)_\#\rho,\rho)}{\eps} \leq |\pa \calW|(\rho) \|\xi\|_{L^2(\rho)},
$$
where we use the inequality $d_{W2}(S_\#\rho,T_\#\rho) \leq \|S-T\|_{L^2(\rho)}$ for the
last inequality. By the same token with $-\xi$ instead of $\xi$, we deduce
$$
\Big|\int_\RR \achapo_u(x)\xi(x)\,d\rho(x)\Big|
\leq |\pa \calW|(\rho) \|\xi\|_{L^2(\rho)}.
$$
Since $\xi$ is arbitrary, we get $\|\achapo_u\|_{L^2(\rho)}\leq |\pa\calW|(\rho)$. 
We conclude by using \eqref{slopmin}.

$\bullet$ {\bf Well-posedness and convergence of the JKO scheme.}

The JKO scheme is defined in \eqref{JKO} where $\tau>0$ is a given small time step. 
Thus, using \eqref{dWF-1} and \eqref{eq:WL2}, the minimization problem 
defined in \eqref{JKO} is equivalent to~: 
let $v_k^{\tau}\in L^2(-\frac 12,\frac 12)$, we define
\beq\label{JKObis}
v_{k+1}^{\tau}= \underset{v\in L^2(-1/2,1/2), \pa_zv\leq 0}{\arg\min} \Big\{
\widetilde{\calW}(v) +\frac{1}{2\tau} \|v-v_k^\tau\|^2 \Big\}.
\eeq
We recall that $a$ is a nondecreasing function, then 
$a(-\frac 12)\leq a(z) \leq a(\frac 12)$ when $z\in (-\frac 12,\frac 12)$.
Then the functional defined inside the brackets is clearly lower 
semi-continuous, convex and coercive on $L^2(-\frac 12,\frac 12)$. 
We deduce that the above minimization problem \eqref{JKObis} 
admits an unique solution.
Moreover, computing the Fr\'echet derivative of the functional 
defining \eqref{JKObis} in the $L^2$-metric, the Euler-Lagrange
equation associated to this minimization problem implies
$$
v^\tau_{k+1}(z) = v_k^\tau(z) - \tau a(z), 
\qquad z\in \big(-\frac 12,\frac 12 \big).
$$
This is an implicit Euler discretization of the equation 
$\pa_t v(t,z) + a(z)=0$. 
The function $a$ being nondecreasing, it is clear that if $v_k^\tau$
is nonincreasing, then $v_{k+1}^\tau$ is nonincreasing.
It is well known that the piecewise
constant interpolation $v^\tau$ defined by $v^\tau(0)=v^{ini}$ 
and $v^\tau(t)=v_k^\tau$ if $t\in (k\tau,(k+1)\tau]$, converges
in $L^2(-\frac 12,\frac 12)$ as $\tau\to 0$ towards $v(t)$ for all 
$t\in [0,+\infty)$, where $v(t,z)= v^{ini}-t a(z)$.
Moreover, we have the energy estimate
$$
\widetilde{\calW}(v) = \widetilde{\calW}(v^{ini}) - 
\int_0^t \int_{-1/2}^{1/2} |a(z)|^2\,dzds.
$$
We define then $\rho:=-\pa_x u$, where 
$u(t,x):=v^{-1}(t,z)=\sup\{z\in [-\frac 12,\frac 12] /~ v(t,z)<x\}$. 
We have $\rho \in \calP_2(\RR)$, and decomposing the integral
between the regular part and the jump part as above, we state that for $s\in (0,t)$,
$$
\int_{-1/2}^{1/2} |a(z)|^2\,dz = \int_\RR |\achapo_u(s,x))|^2 \rho(s,dx).
$$
We deduce from the latter energy estimate,
\beq\label{nrjestim_anonid}
\calW(\rho)=\calW(\rho^{ini}) - \int_0^t \int_\RR |\achapo_u(s,x))|^2 \rho(s,dx).
\eeq
Moreover, we have from the equation $\pa_t v(t,z)+a(z)=0$ that 
the generalized inverse $u$ solves in the weak sense
$\pa_t u + \achapo_u\pa_x u =0$. Therefore, in the sense of distributions, we have 
$$
\pa_t \rho + \pa_x(\achapo_u\rho)= 0.
$$

From Theorem 8.3.1 of \cite{Ambrosio}, we deduce that 
$\rho\in AC_{loc}^2((0,+\infty);\calP_2(\RR))$.
We conclude, using moreover \eqref{nrjestim_anonid} that
$\rho$ is a curve of maximal slope for the functional $\calW$ 
defined in \eqref{Wanonid}.

$\bullet$ {\bf Gradient flow solutions.}

This last step is a direct consequence of Theorem 11.1.3 of \cite{Ambrosio}
since all assumptions of the Theorem have been verified above.
Thus curves of maximal slope are gradient flow solutions.
Uniqueness is obtained thanks to a contraction estimate based on
a Gronwall argument as in \cite{Carrillo}.
This concludes the proof of the first item of Theorem \ref{th:GFanonid}.

$(ii)$ From Proposition \ref{AC2}, we deduce that the duality solution of
Theorem \ref{dual_anonid} belongs to $AC_{loc}^2((0,+\infty);\calP_2(\RR))$.
Then by uniqueness of duality solutions and gradient flow solutions,
we deduce that both notion of solutions coincide.
Moreover, from Proposition \ref{entropy} with $w=0$, equation \eqref{eq:U}
reduces to
$$
\pa_t u + \pa_x A(u) = 0.
$$
It is well-known that this scalar conservation law admits an unique nonincreasing solution
which is the entropy solution (see e.g. \cite[Lemma 3.3]{BJpg}).
Then $\rho=-\pa_xu$.
\end{proof}

\newpage

\appendix

\section*{Technical Lemma}

In this appendix we state a technical lemma which is used in the paper.
\begin{lemma}\label{lemA}
Let us assume that $W$ satisfies assumptions {\bf (A0)--(A4)}.
Let $(\rho_n)_{n\in \NN}$ be a sequence of measures such that
$\rho_n \rightharpoonup \rho$ weakly in $\calM_b(\RR)$. Then
$$
\lim_{n\to +\infty} \int_{x\neq y} W'(x-y) \rho_n(dy) = 
\int_{x\neq y} W'(x-y) \rho(dy), \quad \mbox{ for a.e. } x \in \RR.
$$
\end{lemma}
\begin{proof}
We consider a regularization of $W$ by $W_\eps$ such that for all $\eps>0$, 
$W_\eps \in C^1(\RR^d)$, $W_\eps(-x)=W_\eps(x)$, 
$W_\eps$ and $W'_\eps$ uniformly bounded with respect to $\eps$, and
\beq\label{boundWapprox}
\sup_{x\in \RR\setminus (-\eps,\eps)} |W'_\eps(x)-W'(x)|\leq \eps.
\eeq
By definition of the weak convergence, we have
\beq\label{ineqA1}
\lim_{n\to +\infty} \int_{x\neq y} W'_\eps(x-y) \rho_n(dy) = 
\int_{x\neq y} W'_\eps(x-y) \rho(dy), \quad \mbox{ for a.e. } x \in \RR.
\eeq
In fact, we can remove the point $y=x$ in the integral since by construction
$W'_\eps$ is odd, then $W'_\eps(0)=0$.
Moreover for all $n\in \NN$, we have that
$$
\begin{array}{ll}
\ds \Big|\int_{x\neq y} (W'_\eps-W')(x-y) \rho_n(dy)\Big| = &\ds
\Big|\int_{(x-\eps,x+\eps)\setminus\{x\}} (W'_\eps-W')(x-y) \rho_n(dy)\Big|  \\[2mm]
&\ds +\Big|\int_{\RR \setminus (x-\eps,x+\eps)} (W'_\eps-W')(x-y) \rho_n(dy)\Big|.
\end{array}
$$
We can bound the first term of the right hand side by 
$C \rho_n((x-\eps,x+\eps)\setminus\{x\}) \to 0$ when $\eps \to 0$, where $C$ stand
for a nonnegative constant.
For the second term, we use the property \eqref{boundWapprox} to state the 
convergence towards $0$ as $\eps\to 0$. We deduce that 
$$
\lim_{\eps\to 0} \int_{x\neq y} W'_\eps(x-y) \rho_n(dy) = 
\int_{x\neq y} W'(x-y) \rho_n(dy),
$$
and by the same token with $\rho$ instead of $\rho_n$, we have
$$
\lim_{\eps\to 0} \int_{x\neq y} W'_\eps(x-y) \rho(dy) = 
\int_{x\neq y} W'(x-y) \rho(dy).
$$
We conclude by passing into the limit $\eps\to 0$ in \eqref{ineqA1}.
\end{proof}

{\bf Acknowledgement.} The authors acknowledge warmly the anonymous referee
of this paper for his/her valuable comments that allowed us to improve this
work. The second author acknowledges partial support from the ANR Programme Blanc
project KIBORD, \#ANR-13-BS01-0004.


\end{document}